  \documentclass[10pt, a4paper]{article}
  \pdfoutput=1

  \usepackage{amsmath,amsthm,amssymb}
  \usepackage{graphicx,color}
  \usepackage{dsfont}
  \usepackage{float}
  \usepackage{caption}
  \usepackage{subcaption}
  \usepackage[]{algorithm2e}

  \usepackage{makecell}
  \usepackage{gnuplottex}
  \usepackage{pgfplots}

  \usepackage{accents}

   \usepackage{gnuplottex} 

\newtheorem{thm}{Theorem}[section]
\theoremstyle{plain}

\newtheorem{lemma}[thm]{{\textbf Lemma}}
\newtheorem{theorem}[thm]{{\textbf Theorem}}

\newtheorem{corollary}[thm]{{\textbf Corollary}}
\newtheorem{assumption}[thm]{{\textbf Assumption}}

\newtheorem{definition}[thm]{{\textbf Definition}}
\newtheorem{remark}[thm]{{\textbf Remark}}

\newtheorem{example}[thm]{{\textbf Example}}

\newcommand{\overbar}[1]{\mkern 1.5mu\overline{\mkern-1.5mu#1\mkern-1.5mu}\mkern 1.5mu}

\makeatletter
  \renewcommand*\env@matrix[1][*\c@MaxMatrixCols c]{%
    \hskip -\arraycolsep
    \let\@ifnextchar\new@ifnextchar
  \array{#1}}
\makeatother

\DeclareMathOperator{\divv}{div}
\DeclareMathOperator{\dt}{\frac{d}{dt} }

\providecommand{\keywords}[1]{\textbf{\textit{Keywords: }} #1}
\providecommand{\AMS}[1]{\textbf{\textit{AMS classification: }} #1}

\newcommand{\ac}{\accentset{\circ}}

\newcommand{\Neh}{N^h_\eps}

\newcommand{\R}{\mathbf R}
\newcommand{\N}{\mathbf N}
\newcommand{\Z}{\mathbf Z}

\newcommand{\supp}{\text{supp}}
\newcommand{\eps}{\epsilon}

\newcommand{\Sb}{\textbf{S}}
\newcommand{\Tb}{\textbf{T}}
\newcommand{\gb}{\textbf{g}}

\newcommand{\beqn}{\begin{eqnarray*}}
\newcommand{\eeqn}{\end{eqnarray*}}

\newcommand{\ben}{\begin{equation}}
\newcommand{\een}{\end{equation}}

\newcommand{\beq}{\begin{eqnarray}}
\newcommand{\eeq}{\end{eqnarray}}

\newcommand{\benn}{\begin{equation*}}
\newcommand{\eenn}{\end{equation*}}

\title{Shape optimisation with nonsmooth cost functions: from theory 
	to numerics}

\author{Kevin Sturm\thanks{Universit\"at Duisburg-Essen, Fakult\"at f\"ur Mathematik, Thea-Leymann-Str. 9, D-45127 Essen, Germany (kevin.sturm@uni-due.de)} }
 \date{}
\begin{document}
 
 \maketitle
 

\begin{abstract}
This paper is concerned with the study of a class of nonsmooth cost functions 
subject to a quasi-linear PDE in Lipschitz domains in dimension two. We derive the 
Eulerian semi-derivative of the cost function by employing the averaged 
adjoint approach and maximal elliptic regularity. 
Furthermore we characterise 
stationary points and show how to compute steepest descent directions 
theoretically and practically. 
Finally, we present some numerical results for a 
simple toy problem and compare them with the smooth case.  We also 
compare the convergence rates and obtain higher rates in the nonsmooth case.
\end{abstract}

\keywords{
    shape optimization, nonsmooth cost functions, PDE constraints
}

\AMS{
    49Q10, 49Q12, 35J62, 49K20, 49K40, 49J52
}

\section*{Introduction}
The main object of shape optimisation is the minimisation of a cost or shape function
with respect to a design variable. In applications the design variable may be 
the bodywork of a car or aircraft, but also the shape of antennas or 
inductor coils are possible design variables. The shape function may be the compliance, drag, friction or 
any other physically 
relevant quantities. 
Mathematically speaking the \emph{design variable} 
is a subset of the Euclidean space admitting a certain regularity reflecting
the smoothness of the design and a \emph{shape/cost function} is a 
real-valued mapping on the design variables. 
\\[0.3cm]
While there exists a huge body of research on the topic of smooth shape optimisation 
problems, see \cite{MR2731611,SokZol89,HenPie05,MR1969772} and references therein, 
the work on nonsmooth problems is far less complete.
By a smooth shape optimisation problem we understand that the cost function and the
constraints (usually partial differential equations) are smooth in the sense that the resulting 
Eulerian semi-derivative of the cost function is linear. 
Accordingly we speak of nonsmooth problems when the 
Eulerian semi-derivative is nonlinear.  The nonlinearity can 
have two reasons: the first one is that the 
constraint itself is nonlinear, for instance
it is a variational inequality of first or second kind; \cite{HL11,MR895743,sokosolo,HeineSturm16}. 
The second and more obvious reason for the nonlinearity of the Eulerian semi-derivative
is that the cost function itself
is only directional differentiable which results in a nonlinearity 
of the Eulerian semi-derivative. 
\\[0.3cm]
In this work we focus on nonsmooth cost functions in the
aforementioned sense. To be more precise our cost function is 
maximum of a continuously differentiable function acting on continuous functions subject to 
a nonlinear PDE supplemented with mixed boundary conditions.  
This type of cost function can be used in various applications, such as 
mechanics, free boundary problems and electrical impedance tomography. 
\\[0.3cm]
It is noteworthy that
our approach has similarities to optimal control problems with pointwise state 
constraints; see \cite{MR861100}. We also refer to the work 
\cite{MR2987902,MR2804648,MR1892589} for optimal control problems with
$L_\infty$ cost function.  From the shape optimisation point of view our work 
is related to \cite{MR1641286} where the square of the maximum norm subject to 
the (linear) Helmhotz equations was studied. The authors use the material derivative approach in conjunction
with the notion of subgradient. Our 
results make use of the averaged adjoint 
approach \cite{MR3374631} and the notion of Eulerian semi-derivative which allows the 
derivation of an optimality system under fairly
general assumptions even with quasi-linear state equation. 
\\[0.3cm]
A particularity of our approach, in contrast to previous ones \cite{HO2,HO3}, is
that we follow the paradigm first optimise-then-discretise. One main 
difficulty of our setting is that the partial differential equation is defined 
on a Lipschitz domain and supplemented with mixed boundary conditions for which no higher 
differentiability of the solution can be expected. 
In order to derive the Eulerian semi-differentiability we 
make use of maximal elliptic regularity results and combine them 
with the averaged adjoint approach from \cite{MR3374631,lauraindistributed}. Surprisingly 
also in this nonsmooth situation we can bypass the differentiation 
of the control-to-solution mapping by proving a weak Danskin-type theorem. 
The obtained Eulerian semi-derivative is then further studied 
in an infinite dimensional configuration by using
 valued reproducing kernel  Hilbert 
spaces (vvRKHS). The effectiveness of vvRKHS for smooth shape optimisation problems
has already been presented in \cite{eigelsturm16}. This allows us to carry over results from the 
classical work \cite{MR1088479}.
\subsubsection*{Structure of the paper}
In Section~\ref{sec:preliminary}, we recall basic facts from shape calculus and
results on maximal elliptic regularity in dimension two. 

In Section~\ref{sec:main_maximum_function}, we formulate the problem that is
studied in the subsequent sections. We establish sensitivity results for 
a quasi-linear elliptic PDE with mixed boundary conditions. Furthermore the
Eulerian semi-differentiability of a nonsmooth maximum-type cost is established using the averaged adjoint approach. 

In Section~\ref{sec:charac_stationary_points}, we study properties of the 
Eulerian semi-derivative and prove the existence of steepest descent directions
and $\eps$-steepest descent directions. We then 
propose a discretisation of $\eps$-steepest descent directions adapted for 
finite elements. 

In the final Section~\ref{sec:numerics} we provide numerical experiments validating
our theoretical findings. For that purpose a simple linear PDE
with homogeneous Dirichlet boundary conditions is examined for which 
an analytical solution is available. We compare the results of the 
nonsmooth cost function with a $L_2$-type smooth cost function in order
to highlight the difference.

\section{Preliminaries}\label{sec:preliminary}
\setcounter{equation}{0}
In this section, we recall some basics from shape calculus and PDE theory. For an in-depth 
treatment we refer the reader to the 
monographs \cite{MR2731611,SokZol89,HenPie05,MR1969772}. Numerous examples of 
PDE constrained shape functions and their shape derivatives can be found in \cite{sturm2015shape}. 

\subsection{Sobolev spaces and Gr\"oger regular domains}
We consider special subsets $\Omega\subset\R^2$ satisfying the following 
conditions. 

\begin{definition}[\cite{MR2551487}]\label{def:groeger}
Let $\Omega\subset \R^2$ and $\Gamma\subset \partial \Omega$ be given. 
We say that $\Omega\cup \Gamma$ is regular (in the sense of Gr\"oger) 
if $\Omega$ is a bounded Lipschitz domain, $\Gamma$ is a relatively open part of the boundary $\partial \Omega$,
$\Gamma_0 := \partial \Omega \setminus \Gamma$ 
has positive measure and
$\Gamma_0$ is the finite union of closed and non-degenerated curved pieces
of $\partial \Omega$.
\end{definition}


\begin{remark}
	For higher dimensions the previous definition can be extended via bi-lipschitz
	charts; cf. \cite[Definition 2]{MR2551487}.
\end{remark}

With $\Omega$, $\Gamma$ and $\Gamma_0$ defined as in Definition~\ref{def:groeger}, 
we introduce for $d\ge  1$,
  \begin{align*}
       	C^\infty_c(\Omega,\R^d) & := \{  f|_{\Omega}:\; f\in C^\infty(\R^2,\R^d), \; \supp f\cap \partial \Omega =\emptyset \},\\
        C^\infty_\Gamma(\Omega,\R^d) & := \{  f|_{\Omega}:\; f\in C^\infty(\R^2,\R^d), \; \supp f\cap \Gamma_0 =\emptyset \},\\
        C_\Gamma(\Omega,\R^d) & := \{f: \; f\in C(\overbar{\Omega},\R^d),\; f=0 
		\;\text{ on } \Gamma_0\}.
\end{align*}
In the scalar valued case, that is, $d=1$, we omit the last argument, for 
instance, we write $ C^\infty_c(\Omega) := C^\infty_c(\Omega,\R^1)$. 
If we denote by $\mathcal M(\Omega)$ the space of regular Borel measures, then 
by Riesz representation theorem 
$\mathcal M(\overbar \Omega) \simeq (C(\overbar \Omega))^* $ and also 
$\mathcal M(\Omega\cup \Gamma) \simeq (C_\Gamma(\Omega))^*  $. 

For all finite integers $p,p'\ge 1$ with $1/p+1/p'=1$, we define the Sobolev space 
\ben
W^1_{\Gamma,p}(\Omega,\R^d) = 
\overline{C^\infty_\Gamma(\Omega,\R^d)}^{W^1_p}, 
\qquad  W^{-1}_{\Gamma,p}(\Omega,\R^d) := (W^1_{\Gamma,p'}(\Omega,\R^d))^*. 
\een
In case $\Gamma=\emptyset$ we write 
$ \ac W^1_{p}(\Omega,\R^d) :=  W^1_{\Gamma,p}(\Omega,\R^d)$.
In the scalar valued case we set $W^1_{\Gamma,p}(\Omega) :=  W^1_{\Gamma,p}(\Omega,\R^1)$  and similarly for the
other spaces. 
In case $p=2$ we the use the notation 
$W^1_{\Gamma,2}(\Omega,\R^d) =: H^1_\Gamma(\Omega,\R^d) $ and in case 
$\Gamma=\emptyset$ also 
$ \ac H^1(\Omega,\R^d) := W^1_{\Gamma,2}(\Omega,\R^d) $.

\subsection{Maximal elliptic regularity}
Let $\Omega$, $\Gamma$ and $\Gamma_0$ be as in Definition~\ref{def:groeger}.
Fix $2\le q < \infty$ and denote by $q'$ the conjugate of $q$ defined by
$1/q+1/q'=1$.  Let $b:\overbar \Omega\times \R^3\rightarrow \R^3$ be a function 
satisfying for all $\eta,\theta\in \R^3$ and all $x\in \overbar\Omega$:
\begin{equation}\label{ass:b}
	\begin{split}
	b(\cdot,0)\in L_q(\Omega)  & \text{ and }b(\cdot,\eta)  \text{ is 
		measurable}, \\
	   (b(x,\eta)-b(x,\theta)) &\cdot (\eta-\theta) \ge m|\eta-\theta|^2, \; m>0,\\
          |b(x,\eta)-b(x,\theta)|& \le M|\eta -\theta|, M>0,
        \end{split}
\end{equation}
where $|\cdot|$ denotes the Euclidean norm. Notice that $m\ge M$. 
Let us denote $Lu:=\begin{pmatrix} 
	 u \\
	 \nabla u\end{pmatrix} $. 
Then we define  
$a(\cdot,\cdot)$ via
$
a:W^1_{\Gamma,q}(\Omega)\times W^1_{\Gamma,q'}(\Omega)\rightarrow \R, (v,w) \mapsto \int_\Omega b(x,Lv(x))\cdot Lw(x)  \; dx	
$
and the corresponding operator $\mathcal A_q$, 
\ben\label{def:a}
\mathcal A_q: W^1_{\Gamma,q}(\Omega) \rightarrow W^{-1}_{\Gamma,q}(\Omega), 
\quad v\mapsto \mathcal A_qv := a(v,\cdot).
\een
Let $\mathcal J$ be defined by 
$\langle \mathcal Ju,v\rangle := \int_\Omega \nabla u\cdot \nabla v + uv \; dx$ for all 
$u,v\in W^1_{\Gamma,2}(\Omega)$. By H\"older's inequality it easily follows that
$\mathcal J:W^1_{\Gamma,p}(\Omega)\rightarrow W^{-1}_{\Gamma,p}(\Omega)$ is well-defined for 
all $p\ge 2$. With the help of the operator $\mathcal J$ we may define
$
M_p := \sup\{\|v\|_{W^1_p(\Omega)}:\; v\in W^1_{\Gamma,p}(\Omega), \|\mathcal Jv\|_{W^{-1}_{\Gamma,p}}\le 1 \}.	
$ It is clear that $M_2=1$. 

Henceforth it is useful to collect all regular domains:
	$\Xi := \{ (\Omega,\Gamma):\;  \Omega\subset \R^2, \Gamma\subset 
		\partial \Omega, \text{ and } \Omega\cup \Gamma \text{ is regular} \}.$
We define $\Omega^\Gamma:=(\Omega, \Gamma)$. 
\begin{definition}
Denote by $R_q$, $2\le q < \infty$, the set of regular domains 
$\Omega^\Gamma\in \Xi$ for which $\mathcal J$ maps $W^1_{\Gamma,q}(\Omega)$ onto 
$W^{-1}_{\Gamma,q}(\Omega)$. 
\end{definition}
The following result is \cite[Lemma 1]{MR990595}.
\begin{lemma}\label{lem:Rq}
Let $\Omega^\Gamma\in R_q$ for some $q>2$. Then $\Omega^\Gamma\in R_p$ for $2\le p\le q$ and $M_q \le M_p^\theta$
if
$
\frac{1}{p} = \frac{(1 - \theta)}{2} + \frac{\theta}{q}.
$
\end{lemma}
\begin{remark}
\begin{itemize}
	\item If $\Omega\subset\R^2$ is a bounded domain of class $C^1$, then 
		$(\Omega,\emptyset)\in \cap_{q\ge 2}R_q$; cf. \cite[Remark 7]{MR990595}. 
	\item For every regular $(\Omega,\Gamma)\in \Xi$ there is $q>2$, so that 
		$ \Omega^\Gamma \in R_q$; cf. \cite[Theorem 3]{MR990595}.
	\item If $\Omega^\Gamma\in R_q$, then $M_q<\infty$.  
\end{itemize}
\end{remark}

We can now state a result showing that the operator $\mathcal A_q$ (in dimension two) is always an 
isomorphism for some (possibly small) $q>2$. We recall the following version of \cite[Theorem 1]{MR990595}.

\begin{theorem}[\cite{MR990595}]\label{thm:groeger}
	Let  $\Omega^\Gamma\in R_{q_0}$, $q_0\ge 2$.  Suppose that $b(\cdot,\cdot)$ satisfies Assumption~\ref{ass:b}
with $q_0$ and let $\mathcal A_q$ be defined by  \eqref{def:a}. Then 
$\mathcal A_q:W^1_{\Gamma,q}(\Omega)\rightarrow W^{-1}_{\Gamma,q}(\Omega)$ is an isomorphism provided that
$q\in [2,q_0]$ and $M_qk<1$, where $k:=(1-m^2/M^2)^{1/2}$.  In that case
\ben\label{eq:A_q}
\|\mathcal A_q^{-1}f-\mathcal A_q^{-1}g\|_{W^1_q(\Omega)} \le c_q\|f-g\|_{W^{-1}_{\Gamma,q}(\Omega)} \quad \text{ for all } f,g\in W^{-1}_{\Gamma,q}(\Omega),
\een
where $c_q:= mM^{-2}M_q (1-M_qk)^{-1} $. Finally, $M_qk<1$ is satisfied if 
\ben\label{eq:1q}
\frac{1}{q} > \frac{1}{2} - \left(\frac{1}{2}- \frac{1}{q_0} \right) \frac{|\log k|}{\log M_{q_0}} .
\een
\end{theorem}

\begin{corollary}
For small $q>2$ the constant $c_q$ in \eqref{thm:groeger} can be chosen to be independent of $q$. 
\end{corollary}
\begin{proof}
Assume first $k = 0$. Then  Lemma~\ref{lem:Rq} shows $M_{q} \le M_{q_0}^\theta$ with $
\frac{1}{q} = \frac{(1 - \theta)}{2} + \frac{\theta}{q_0}$ (or $\theta = \frac{q_0}{q}\frac{q-2}{q_0-2}$).
Therefore 
$c_q \le mM^{-2}M_q \le mM^{-2}M_{q_0}^\theta \le mM^{-2}\max_{q\in [2,q_0]}M_{q_0}^{\theta(q)}  $
and the maximum is attained as $\theta(\cdot)$ is continuous on $[2,q_0]$.

Assume now $k>0$. As shown in \cite{MR990595}, inequality \eqref{eq:1q} 
follows from Lemma~\ref{lem:Rq}. To be more precise the estimate 
$M_{q} \le M_{q_0}^\theta$ with $
\frac{1}{q} = \frac{(1 - \theta)}{2} + \frac{\theta}{q_0}
$ shows that $M_q^\theta k<1$ implies $M_{q_0}k<1$ 
	and indeed elementary computations show that $M_{q_0}^\theta k<1$ is equivalent to \eqref{eq:1q}. In much the same way one can use Lemma~\ref{lem:Rq} to show that 
$M_qk\le 1-\eps$, where $\eps>0$, is satisfied if
\ben\label{eq:1q_eps}
\frac{1}{q} \ge  \frac{1}{2} - \left(\frac{1}{2} - \frac{1}{q_0} \right)\left( \frac{\log(1-\eps)}{\log M_{q_0}} +  \frac{|\log k|}{\log M_{q_0}}\right).
\een
In fact $M^\theta_{q_0}k<1-\eps$ with $\theta = \frac{q_0}{q}\frac{q-2}{q_0-2}$ is equivalent to \eqref{eq:1q_eps}. 
This shows that there is $\eps>0$ so that for all small $q>2$ we have 
$c_q\le mM^{-2}M_q (1-M_qk)^{-1} \le  mM^{-2}(1-\eps)/(k\eps)$ and thus $c_q$ 
in \eqref{eq:A_q} can be replaced by $mM^{-2}(1-\eps)/(k\eps)$ provided $q>2$ 
is small enough.
\end{proof}

\subsection{Shape functions, shape derivative and shape gradients}
Let $D\subset \R^d$, $d\ge 1$, be an open and bounded set.
Given a vector field $X\in \ac C^{0,1}(D,\R^d)$, we denote by $\Phi_t$ the flow of $X$ 
(short $X$-flow) given by $\Phi_t(x_0) := x(t)$, where $x(\cdot)$ solves
\ben\label{eq:flow}
 x'(t) = X(x(t)) \text{ in } 
 (0,\tau), \; x(0)=x_0.	
\een  
The space $\ac C^{0,1}(D, \R^d)$ comprises all bounded and Lipschitz 
continuous functions on $\overbar{D}$ vanishing on $\partial D$. It is 
a closed subspace of $C^{0,1}(D,\R^2)$, the space of bounded Lipschitz continuous mapping
defined on $\overbar D$. 
Similarly we denote by $\ac C^k(D,\R^d)$ all function $k$-times differentiable 
function on $D$ vanishing on $\partial \Omega$. 
Note that by the chain rule (omitting the space variable $x$) 
$(\partial (\Phi^{-1}_t))\circ \Phi_t=(\partial \Phi_t)^{-1}=:\partial \Phi^{-1}_t.$
We denote by $\wp(D)$ the powerset of $D$. 
Let $\Xi\subset \wp(D)$ be given. 
\begin{definition}\label{def:shape_function}
	\begin{itemize}
	\item[(i)] A mapping $J:\Xi\subset \wp(D) \rightarrow \R $ is called \emph{real shape function} or \emph{shape function}.
	\item[(ii)]	
		A mapping $u:\Xi\subset \wp(D) \rightarrow \R^{D}$ 
		 with values in  
		$\R^{D}:=\{f:D\rightarrow \R\}$, is called \emph{abstract shape function}.
	The set $\Xi$ is referred to as \emph{admissible set}. 
          \end{itemize}
\end{definition}

\begin{definition}\label{def1} 
	Let $J:\Xi\subset \wp(D) \rightarrow \R$ a  shape function defined on subsets
	of $D$. Assume that $\mathcal H(D,\R^d)\subset \ac C^1(D,\R^d)$ is 
	a subspace. Let $\Omega\in \Xi$ and $X\in \mathcal H(D,\R^d)$ be such that 
	$\Phi_t(\Omega) \in \Xi$ for all 
	$t >0$ sufficiently small.   
	Then the  \emph{Eulerian semi-derivative} of $J$ at $\Omega$ in direction $X$ is defined by
	\ben
	dJ(\Omega)(X):= \lim_{t \searrow 0}\frac{J(\Phi_t(\Omega))-J(\Omega)}{t} .
	\een
    We say that $J$ is
	\begin{itemize}
		\item[(i)] \emph{Eulerian semi-differentiable} at $
			\Omega$ in $\mathcal H(D,\R^d)$, if 	 $dJ(\Omega)(X)$ 
			exists for all $X \in \mathcal H(D,\R^d)$. 
		\item[(ii)]  \textit{shape differentiable} at $\Omega$ 
			in $\mathcal H(D,\R^d)$ if  $dJ(\Omega)(X)$ exists for all $X \in \mathcal H(D,\R^d)$ and 
		$ X     \mapsto dJ(\Omega)(X) $ 
		is linear and continuous.
	\end{itemize}
\end{definition}

Another auxiliary result that is frequently used is the following:
\begin{lemma}\label{lemma:phit}
    Let $D\subseteq \R^d$ be open and bounded and suppose $X\in \ac C^1(D, \R^d)$.
    \begin{itemize} 
        \item[(i)] We  have
        \begin{align*}
        \frac{ \partial \Phi_t - I}{t} \rightarrow  \partial X \quad 
        &\text{ and }\quad   \frac{ 
            \partial \Phi_t^{-1} - I}{t} \rightarrow  - \partial X && 
        \text{ strongly in } C(\overbar D, \R^{d,d})\\
        \frac{ \det(\partial \Phi_t) - 1}{t} \rightarrow & \divv(X) && \text{ strongly in } C(\overbar D).
        \end{align*}
        \item[(ii)]  	For all open sets $\Omega \subseteq D$ and all 
		$\varphi \in L_p(\Omega)$, $1\le p<\infty$, we have
		\begin{align}
        \varphi\circ \Phi_t  \rightarrow & \varphi 
        && \text{ strongly in }  L_p(\Omega).
        \end{align}
		Moreover, if $\varphi\in W^1_p(\Omega)$, $1\le p<\infty$, then we have 
        \begin{align}
        \frac{\varphi\circ \Phi_t  - \varphi}{t} \rightarrow & \nabla \varphi \cdot X
        && \text{ strongly in }  L_p(\Omega).
        \end{align}
    \end{itemize}
\end{lemma}

Consider a function $J:\Xi\subset \wp(D) \rightarrow \R$ that is shape 
differentiable at $\Omega\in \Xi$ where $D\subset \R^d$. 
Suppose there is a Hilbert space $\mathcal H(\mathcal X,\R^d)$ of functions from 
$\mathcal X\subset D$ into $\R^d$ and assume $dJ(\Omega) \in \mathcal H(\mathcal X,\R^d)^*$. 
\begin{definition}\label{def:H_gradient} 
The gradient of $J$ at $\Omega$ with respect to the space 
$\mathcal{H}(\mathcal X,\R^d)$ and the inner product $(\cdot, \cdot )_{\mathcal{H}(\mathcal X,\R^d)}$, denoted $\nabla J(\Omega)$, is defined by
 \begin{equation}\label{eq:H_gradient}
	dJ(\Omega)(\varphi)=( \nabla J(\Omega), \varphi)_{\mathcal{H}(\mathcal 
		X,\R^d)} \; \text{ for all } \varphi\in \mathcal{H}(\mathcal X,\R^d).
\end{equation}
We also call $\nabla J(\Omega)$ the $\mathcal H(\mathcal X,\R^d)$-gradient of $J$ 
at $\Omega$.
\end{definition}

\subsection{Projections in Hilbert spaces}
Let us recall the following basic result on projections in Hilbert spaces. 
\begin{lemma}\label{lem:hilbert_projection}
	Let $H$ be a real Hilbert space, $K\subset H$ 
	a closed and convex subset and $x_0\in H$. For $x^*\in K$ the following 
	statements are equivalent:
	\begin{itemize}
		\item[(i)] $\|x_0-x^*\|_{H} = \inf_{x\in K}\|x-x_0\|_H$		
		\item[(ii)] $(x_0-x^*,x-x^*)_H \le 0 \quad \text{ for all } x\in K$. 
	\end{itemize}
	Moreover, for each $x_0\in H$ there exists a unique element $x^*\in K$ satisfying $(i)$ (or equivalently $(ii)$). 
\end{lemma}
\begin{proof}
	See \cite[Satz V.3.2, p.219 and Lemma V. 3.3, p.220]{MR1787146}.
\end{proof}
The previous lemma allows to define the projection mapping $P_K:H\rightarrow K$ via 
$P_K(x_0):=x^*$. We also call $x^*=P_K(x_0)$ the projection of $x_0$ on $K$. 
Accordingly $x^*:=P_K(0)$ is the 
point in $K$ that is closest to the origin $0$ and $(ii)$ reads
$
(x^*,x^*)_H \le (x^*,x)_H$ for all  $x\in K. 	
$

\section{Maximum shape function subject to a  quasi-linear PDE}\label{sec:main_maximum_function}
This section is devoted to the derivation of the Eulerian 
semi-differentiability of a nonsmooth shape function subject to a quasi-linear
partial differential equation.  
\subsection{Problem formulation and setting}
Let us fix an open and bounded hold-all set $D\subset \R^2$.
In this paper we study the maximum shape function
\ben\label{eq:cost_max}
J_{\infty}(\Omega^\Gamma) :=  \max_{x\in \overbar \Omega} \Psi(x,u(\Omega^\Gamma,x)), 
\een
where $\Omega^\Gamma = (\Omega,\Gamma)$ belongs to
$
\Xi := \{ (\Omega,\Gamma):\;  \Omega\subset D, \Gamma\subset 
		\partial \Omega, \text{ and } \Omega\cup \Gamma \text{ is regular} \},
$
and $u(\cdot)=u(\Omega^\Gamma,\cdot)$ solves (in a weak sense) the 
following quasi-linear PDE with mixed boundary conditions 
\begin{align}
	- \divv( \beta(|\nabla u|^2)\nabla u ) + u & = f \quad \text{ in } \Omega, \\
	u & = 0 \quad \text{ on } \partial \Omega\setminus \Gamma\; (=: \Gamma_0),\\
		 \partial_\nu u & = 0 \quad \text{ on } \Gamma. 
\end{align}
As usual $\partial_\nu u:=\nabla u\cdot \nu$ is the normal derivative and $\nu$ denotes the outward 
pointing unit normal vector along $\partial \Omega$. 
The functions $\Psi$, $f$, and $\beta$ are specified below. 

Our first task is to prove the Eulerian semi-differentiability of 
$J_\infty(\cdot)$ at sets $\Omega^\Gamma$ belonging to the admissible set
$\Xi$. To emphasise the dependency of $u$ on $\Omega^\Gamma$ we write $u(\Omega^\Gamma,\cdot)$, 
however, we drop the index $\Omega^\Gamma$ whenever no confusion arises.  In what follows it is convenient to introduce the shape function 
\ben\label{eq:cost_j}
j(\Omega^{\Gamma,y}) := \Psi(y,u(\Omega^\Gamma,y)),	
 \een
 depending on the shape variable $\Omega^{\Gamma,y} := (\Omega,\Gamma,y)\in \Xi\times \overbar \Omega$.

To make sense of $J_\infty(\Omega^\Gamma)$ it suffices to have  $u\in W^1_{\Gamma,q}(\Omega)$
with $q>2$ since in that case Sobolev's embedding implies $u\in C_{\Gamma}(\Omega)$.
In order to obtain this higher integrability of $u$ we make the following assumptions. 
\begin{assumption}\label{assmp:trans}
We require the function $\beta:\R\rightarrow \R$ to satisfy the following conditions:
\begin{itemize}
 \item[1.] There exist constants $\bar{\beta},\underline{\beta}>0$ such that 
$\bar{\beta} \le \beta(x)\le \underline{\beta}$ for all $x\in \R.$
 \item[2.] For all $x,y\in \R$, we have $(\beta(x)-\beta(y))(x-y)\ge 0.$
 \item[3.] The function $\beta$ is continuously differentiable, that 
	 is, $\beta\in C^1(\R)$.
 \item[4.]  There are constants $k,K>0$, such that 
	 \ben k |\eta|^2 \le  \beta(|p|^2)|\eta|^2 + 2\beta' ( |p|^2 ) |p \cdot \eta|^2 \le  K |\eta|^2\quad \text{ for all } \; \eta,p \in \R^2. 
	 \een
 \end{itemize}
\end{assumption}

\begin{remark}
	Notice that using (1) and (2) of the previous assumption, we obtain 
	\ben
      	\underbrace{\beta(|p|^2)}_{\ge \bar\beta}|\eta|^2 + 
	2\underbrace{\beta' ( |p|^2 )}_{\ge 0} |p \cdot \eta|^2 \ge  \bar 
	\beta |\eta|^2 \quad \text{ for all }  \eta,p\in \R^2. \een
 So (1) and (2) imply the left inequality in 
 item 4.	
\end{remark}

\begin{assumption}
We assume that $f\in L_q(D)$ for some $q>2$.
\end{assumption}

\begin{assumption}
We assume that the functions $\Psi:\R^2\times\R \rightarrow \R$ satisfies, 
\begin{itemize}
	\item for all $x\in \R^2$, $\Psi(x,\cdot)\in C^1(\R)$ and $\partial_\zeta \Psi\in C(\R^3)$,
	\item for all $\zeta\in \R$, $\Psi(\cdot,\zeta)\in C^1(\R^2)$. 
\end{itemize}
\end{assumption}

\begin{example}
    A typical example of $\Psi$ is the function $\Psi(x,z):= |z  - u_d(x)|^2$, where
    $u_d:\R^2\rightarrow \R$ is some continuously differentiable function. 
    For this choice of cost function we present numerical results in Section~\ref{sec:numerics}. 
\end{example}

\begin{lemma}\label{lem:monoton_continuity}
Let $\beta:\R \rightarrow \R$ satisfy Assumption~\ref{assmp:trans}. 
Then for all $\theta,\eta\in \R^2$,
\begin{align}
\label{eq:monoton}
k|\eta-\theta|^2 &\le   (\beta(|\eta|^2) \eta - \beta(|\theta|^2) \theta)\cdot(\eta-\theta),\\
\label{eq:continuity}   K|\eta-\theta| & \ge |\beta(|\eta|^2) \eta - \beta(|\theta|^2) \theta|.
\end{align}
\begin{proof}
We obtain by the fundamental theorem of calculus,
\ben\label{eq:fundamental}
\begin{split}
	(\beta(|\eta|^2) \eta - \beta(|\theta|^2) \theta)\cdot(\eta-\theta)  = &  \int_0^1 2\beta'(|s\eta +(1-s)\theta|^2) 
	|(\eta-\theta)\cdot(s\eta +(1-s)\theta)|^2\; ds \\
	& + \int_0^1 \beta(|s\eta + (1-s)\theta|^2)|\eta-\theta|^2 \; ds 
	\quad\text{ for all } \theta,\eta\in \R^2.
\end{split}
\een
Hence \eqref{eq:monoton} follows from 
Assumption~\ref{assmp:trans}, item 4. The continuity 
\eqref{eq:continuity} follows in the same way. 
\end{proof}
\end{lemma}

Note that, in general, $u\not\in H^2(\Omega)$ due to the mixed boundary 
conditions. However, we have the following result.   
\begin{lemma}\label{lem:existence}
	Let Assumption~\ref{assmp:trans} be satisfied and assume $\Omega^\Gamma\in 
	\Xi$. For every small enough $q>2$, there is a unique $u\in W^1_{\Gamma,q}(\Omega)$ satisfying
\ben\label{eq:state}
\int_\Omega \beta(|\nabla u|^2)\nabla u \cdot \nabla \varphi + u \varphi \, dx  = \int_\Omega f \varphi \, dx \quad \text{ for all }\varphi\in W^1_{\Gamma,q'}(\Omega)
\een
or equivalently
\ben
\int_\Omega a(x,Lu(x))\cdot L\varphi(x) \; dx = \int_\Omega f \varphi \, dx \quad \text{ for all }\varphi\in W^1_{\Gamma,q'}(\Omega),
\een
where
\ben\label{eq:b_operator} 
a(x,\zeta) := \begin{pmatrix}
       	\zeta_0 \\
      	\beta(|\hat \zeta|^2) \hat \zeta
	\end{pmatrix}, \quad \zeta = \begin{pmatrix}
	\zeta_0\\
	\hat \zeta\end{pmatrix}\in \R^3, \quad Lu:=\begin{pmatrix} 
	 u \\
	 \nabla u\end{pmatrix} .	
\een 
\end{lemma}
\begin{proof}
We apply Theorem~\ref{thm:groeger} to 
$ b(x,\zeta) := a(x,\zeta)$ with $a$ defined in \eqref{eq:b_operator}. 
We need to check the conditions 
stated in \eqref{ass:b}. It is clear that $b(\cdot,0)\in L_\infty(\Omega)$. 
Since Assumption~\ref{assmp:trans} is satisfied,  
Lemma~\ref{lem:monoton_continuity} yields \eqref{eq:monoton} and 
\eqref{eq:continuity} and hence this implies the continuity and monotonicity 
properties for $b(x,\cdot)$ stated in \eqref{ass:b}. Setting $l=(1-(m/M)^2)^{1/2}$ 
and $m := \min\{k,1\}$, $M := \max\{K,1\}$ we see that the condition
$M_ql <1$ is satisfied provided $q>2$ is small enough (cf. \eqref{eq:1q}).
 So the result follows
from Theorem~\ref{thm:groeger}.
\end{proof}

\subsection{Analysis of the perturbed state equation}
Let $\Omega^\Gamma\subset \Xi$ be fixed and pick a vector field 
$X\in \ac C^1(D,\R^2)$ with associated $X$-flow $\Phi_t$. We set 
$\Omega_t := \Phi_t(\Omega)$, $t\ge 0$, and 
consider \eqref{eq:state} on the perturbed domain $\Omega_t$	 and perform a 
change of variables to obtain,
\ben\label{eq:perturbed_ut}
\int_\Omega \beta(|B(t)\nabla u^t|^2)A(t)\nabla u^t \cdot \nabla \varphi + \xi(t) u^t \varphi \, dx = \int_\Omega f^t 
\varphi \, dx \quad \text{ for all } \varphi \in W^1_{\Gamma,q'}(\Omega),
\een
where $q\ge 2$ with its conjugate $ q' = q/(q-1)$, and 
\ben
A(t):= \det(\partial \Phi_t) \partial \Phi_t^{-1}\partial \Phi_t^{-\top}, \quad B(t) := \partial \Phi_t^{-\top}, \quad f^t := \det(\partial \Phi_t) f\circ \Phi_t, \quad \xi(t):= \det(\partial \Phi_t). 	
\een
The existence and uniqueness of a solution of \eqref{eq:perturbed_ut} is 
addressed below. It is convenient to rewrite \eqref{eq:perturbed_ut} as
\ben
\int_\Omega a^t(x,Lu^t(x))\cdot L\varphi(x)\; dx = \int_\Omega f \varphi \; dx	\quad \text{ for all } \varphi \in W^1_{\Gamma,q}(\Omega)
\een
with the definition 
\ben
a^t(x,\zeta) := \begin{pmatrix} \xi(t,x)\zeta_0 \\ \beta(|B(t,x)\hat \zeta|^2) 
	A(t,x)\hat \zeta  \end{pmatrix},\quad  \zeta = \begin{pmatrix} \zeta_0 \\ \hat\zeta \end{pmatrix} \in \R^3.
\een
We associate with $a^t$ the operator
\ben\label{eq:operators_At}
\mathcal A^t_q:W^1_{\Gamma,q}(\Omega) \rightarrow W^{-1}_{\Gamma,q}(\Omega),\; 	
\langle \mathcal A^t_qv,w\rangle := \int_\Omega a^t(x, Lv(x)) \cdot Lw(x) \, 
dx, 
\een
where $q>2$. We show next that for all sufficiently small
$q>2$ and $t>0$ the operators $\mathcal A^t_q$ 
are isomorphisms from $W^1_{\Gamma,q}(\Omega)$ onto $W^{-1}_{\Gamma,q}(\Omega)$. 
The main task is to show that $q$ is independent of $t$ provided 
it is small 
enough. We begin with the following lemma. 

\begin{lemma}\label{lem:postive_At}
For every $\eps> 0$, there exists $\delta>0$, so that,
\begin{align}
\label{eq:A_positive} A(t,x)\eta \cdot \eta  & \ge (1-\eps) |\eta|^2 && \text{ for all } 
\eta\in \R^d, \text{ for all } (t,x)\in [0,\delta]\times \overbar D,\\
\label{eq:A_bound}
\|A(t)\|_{C(\overbar D,\R^{d,d})}   & \le 1+\eps && \text{ for all } t\in [0,\delta],\\
\label{eq:Bt} 1-\eps \le |B(t,x)\eta | & \le 1+\eps && \text{ for all }\eta\in \R^d   \text{ for all } (t,x)\in [0,\delta]\times \overbar D,\\
\label{eq:xit} 1-\eps \le \xi(t,x) & \le  1+\eps && \text{ for all } (t,x)\in 
[0,\delta]\times \overbar D. 
\end{align}
\end{lemma}
\begin{proof}
	We only prove \eqref{eq:Bt} as the other estimates can be shown in much the same way.
	Since $B:[0,\tau]\rightarrow C(\overbar D,\R^{d,d})
	$ is continuous and $B(0)=I$, we find for every $\eps>0$ a number $\delta>0$ so that
$
       \|B(t)-I\|_{C(\overbar D,\R^{d,d})} \le \eps
	$
    for all $|t|\le \delta$.
	Hence the left inequality in \eqref{eq:Bt} follows by the reverse 
	triangle inequality. As for the right inequality in  \eqref{eq:Bt}
	note that for all $\eta\in \R^d$ and all $(t,x)\in [0,\delta]\times \overbar D$, 
	\ben
	\begin{split}
		|\eta|^2 & = |(I-B(t,x))\cdot \eta| + |B(t,x)\eta| \\
	          & \le \underbrace{\|B(t)-I\|_{C(\overbar D,\R^{d,d})}}_{\le \eps}|\eta|^2 + |B(t,x)\eta|
		   \le \eps|\eta|^2 +  |B(t,x)\eta|  
        \end{split}
	\een
	which is equivalent to \eqref{eq:Bt}.
\end{proof}
\begin{lemma}\label{lem:invertibility}
	For each $\Omega^\Gamma\in \Xi$, there exist $q_0>2$ and $\delta>0$, so that for all $t\in [0,\delta]$ and all 
	 $q\in [2,q_0]$ the mapping 
	 $\mathcal A^t_q:W^1_{\Gamma,q}(\Omega) \rightarrow 
	 W^{-1}_{\Gamma,q}(\Omega)
	 $  is an isomorphism. Moreover, there is a constant $c>0$ independent of $t$, so that
	 \ben
	 \|(\mathcal A^t_q)^{-1}f - (\mathcal A^t_q)^{-1}g\|_{W^1_q(\Omega)} \le c\|f-g\|_{W^{-1}_{\Gamma,q}(\Omega)} \quad \text{ for all } f,g\in W^{-1}_{\Gamma,q}(\Omega)
	 \een
	 for all $t\in [0,\delta]$. 
 \end{lemma}
 \begin{proof}
	 According to Theorem~\ref{thm:groeger} and Lemma~\ref{lem:existence}
	 there is $q_0>2$ so that  
	 $\mathcal A(\cdot) = - \divv(\beta(|\nabla \cdot|^2)\nabla \cdot )$ is an ismorphism from $W^1_{\Gamma,q}(\Omega)$ onto 
	 $W^{-1}_{\Gamma,q}(\Omega)$ for all $q\in (2,q_0]$. Indeed setting
	 $m := \min\{k,1\}$ and $M:=\max\{K,1\}$ we get $M_q (1-m^2/M^2)^{1/2} <1$ provided
	 $q$ is close enough to $2$ (cf. \eqref{eq:1q}).  
 Similarly to 
\eqref{eq:fundamental}, we can write 
\ben
\begin{split}
	(\beta(|B(t) & \eta|^2) A(t)\eta - \beta(|B(t)\theta|^2) A(t)\theta)\cdot(\eta-\theta) \\
	&	=   \int_0^1 \xi(t) 2\beta'(|B(t)(s\eta +(1-s)\theta)|^2) 
       |B(t)(\eta-\theta)\cdot B(t)(s\eta +(1-s)\theta)|^2\; ds \\
       & + \int_0^1 \xi(t)\beta(|B(t)(s\eta + (1-s)\theta)|^2)|B(t)(\eta-\theta)|^2 \; ds 
       \quad\text{ for all } \theta,\eta\in \R^2, \text{ for all }t.
\end{split}
\een
So using Assumption~\ref{assmp:trans} and Lemma~\ref{lem:postive_At}, we get
\ben
(\beta(|B(t)\eta|^2) A(t)\eta - \beta(|B(t)\theta|^2) A(t)\theta)\cdot(\eta-\theta) \ge (m-\eps) |B(t)(\eta-\theta)|^2 \ge (k-\eps)|\eta-\theta|^2 	
\een
for all $\theta,\eta\in \R^2$ and all sufficiently small $t$. In a similar manner we can show
\ben
|\beta(|B(t)\eta|^2) A(t)\eta - \beta(|B(t)\theta|^2) A(t)\theta)| \le (M +\eps)|\eta-\theta|	
\een
for all $\theta,\eta\in \R^2$ and all sufficiently small $t$.  
This implies that we find for $\eps >0$ a number $\delta >0$ so that
\begin{align} 
	(a^t(x,\eta)-a^t(x,\theta)) &\cdot (\eta-\theta) \ge (m-\eps)|\eta-\theta|^2, \\
	   |a^t(x,\eta)-a^t(x,\theta)|& \le (M+\eps)|\eta -\theta|
 \end{align}
for all $t\in [0,\delta]$ and for all $\theta,\eta\in \R^{3}$. 
Noting that $a^t(\cdot,0)\in L_\infty(D)$ we can apply again  Theorem~\ref{thm:groeger}
and obtain that $\mathcal A^t_q $ is in fact an isomorphism when we choose 
$\eps$ so small that $M_q(1-(m-\eps)^2/(M+\eps)^2)^{1/2} <1 $ which is 
possible since 
$\lim_{\eps\searrow 0} (m-\eps)/(M+\eps)= m /M$ 
and $M_q  (1 - m^2 /M^2)^{1/2}<1$. 
\end{proof}
\begin{definition}
	For $\Omega^\Gamma\in \Xi$ we define  $q_0>2$ to be a number as in Lemma~\ref{lem:invertibility}.
\end{definition}
\begin{corollary}\label{cor:sensi_u_ut}
	Suppose that $X\in \ac C^1(D,\R^2), \Omega^\Gamma\in \Xi$ and $q\in 
	(2,q_0]$.  
	\begin{itemize}
		\item[(a)] If $f\in 
	L_q(\Omega)$, then the family 
	of solutions $\{u^t\}$ of \eqref{eq:perturbed_ut} satisfies
	\ben
	\lim_{t\searrow 0} \|u^t-u\|_{W^1_q(\Omega)}=0 \quad \text{ and } \quad \lim_{t\searrow 0} \|u^t-u\|_{C(\overbar \Omega)}=0. 
	\een	
\item[(b)]If $f\in W^1_q(\Omega)$, then there is $\tau> 
	0$ and $c>0$, so that $\{u^t\}$ satisfies
	\ben
       	\|u^t-u\|_{C(\overbar{\Omega})} +  \|u^t-u\|_{W^1_q(\Omega)} \le c t 
	\quad 	\forall t\in [0,\tau].	 
	\een
\end{itemize}	
\end{corollary}
\begin{proof}
	Let us first show $(a)$. By Lemma~\ref{lem:invertibility} we find 
	$\delta>0$ and $q_0>2$, so that 
	$ 
	\|u^t\|_{W^1_q(\Omega)} = \|(\mathcal A^t_q)^{-1}f^t\|_{W^1_q(\Omega)} \le c\|f^t\|_{W^{-1}_{\Gamma,q}(\Omega)} 
$
for all $q\in (2,q_0]$ and all $t\in [0,\delta]$	
and using H\"older's inequality the right hand side can be further estimated
\ben
\|f^t\|_{W^{-1}_{\Gamma,q}(\Omega)} = \sup_{\substack{  \varphi\in W^1_{\Gamma,q'}(\Omega) \\ \|\varphi\|_{W^1_{q'}\le 1 }}} \left|\int_\Omega f^t \varphi \; dx\right| \le \|f^t\|_{L_q(\Omega)}.  
\een
The boundedness of $\|f^t\|_{L_q(\Omega)}$ follows from Lemma~\ref{lemma:phit}.
So $u^t$ is bounded in $W^1_{\Gamma,q}(\Omega)$ with $q\in (2,q_0]$. Now by definition $u^t$ and 
$u:=u^0$ satisfy (setting $\mathcal A_q := \mathcal A^0_q$) the operator equations
$
\mathcal A^t_q u^t = f^t
$ and  
$
\mathcal A_q u = f. 
$
Therefore the difference $z^t:=u^t-u$ solves
$
\mathcal A_qz^t =- (\mathcal A^t_q-\mathcal A_q)u^t - (f^t-f) \in W^{-1}_{\Gamma,q}(\Omega)
$
and hence using again Lemma~\ref{lem:invertibility}
  we find $c>0$ so that for all $t$,
\ben\label{eq:estimate_zt}
\begin{split}
	\|z^t\|_{W^1_q(\Omega)} & \le c\|- (\mathcal A^t_q-\mathcal A_q)u^t - (f^t-f)\|_{W^{-1}_{\Gamma,q}(\Omega)} \\
	& \le c(\|(\mathcal A^t_q-\mathcal A_q)u^t\|_{W^{-1}_{\Gamma,q}(\Omega)} +    \|f^t-f\|_{W^{-1}_{\Gamma,q}(\Omega)}).
\end{split}
\een
Furthermore we have for a.e. $x\in \Omega$ and all $t\in [0,\delta]$,
\ben\label{eq:estimate_op_At}
\begin{split}
	|\beta(|B(t,x)\nabla u^t(x) & |^2)A(t,x)\nabla u^t(x) - \beta(|\nabla u^t(x)|^2)\nabla u^t(x)| 
\le \\
& \underbrace{| A(t,x) -  \xi(t,x)B^\top(t,x) |}_{\le 
	ct, \text{ by } Lemma~\ref{lemma:phit}, 
	(i)}\underbrace{|\beta(|B(t,x)\nabla u^t(x)|^2)\nabla u^t(x)|}_{\le c 
	|\nabla u^t(x)|, \text{ by } Lemma~\ref{lem:postive_At}}  \\
& + \underbrace{|\beta(|B(t,x)\nabla u^t(x)|^2)B(t,x)\nabla u^t(x) - \beta(|\nabla u^t(x)|^2)\nabla u^t(x)|}_{\le K|B(t,x)-I||\nabla 
	u^t(x)|, \text{ by }\eqref{eq:continuity}} \le ct |\nabla u^t(x)|. 
\end{split}
\een
So using again H\"older's inequality yields
\ben\label{eq:At_A}
\begin{split}
\|(&\mathcal A^t_q- \mathcal A_q)u^t\|_{W^{-1}_{\Gamma,q}(\Omega)} \\ 
& = \sup_{\substack{  \varphi\in W^1_{\Gamma,q'}(\Omega) \\ 
		\|\varphi\|_{W^1_{q'}\le 1 }}} \left|\int_\Omega 
\underbrace{(\beta(|B(t)\nabla u^t|^2)A(t)\nabla u^t - \beta(|\nabla 
	u^t|^2)\nabla u^t)\cdot \nabla \varphi }_{\le ct |\nabla 
	u^t||\nabla \varphi|, \text{ by }\eqref{eq:estimate_op_At}}+ 
\underbrace{(\xi(t)-1)}_{ \le ct, \text{ by } Lemma~\ref{lemma:phit}, (i)}u^t 
	\varphi \; dx \right|\\
	\le & ct\bigg( \sup_{\substack{  \varphi\in 
		W^1_{\Gamma,q'}(\Omega) \\ \|\varphi\|_{W^1_{q'}\le 1 
		}}}\int_\Omega |\nabla u^t| |  \nabla \varphi| \; dx + 
 \sup_{\substack{  \varphi\in 
		 W^1_{\Gamma,q'}(\Omega) \\ \|\varphi\|_{W^1_{q'}\le 1 }}}
 \int_\Omega |u^t \varphi|\; dx\bigg)\le   c t \underbrace{\|u^t\|_{W^1_q(\Omega)}}_{\le c} \le c t 
\end{split}
\een
and similarly
\ben\label{eq:f_ft}
\|f^t-f\|_{W^{-1}_{\Gamma,q}(\Omega)} \le  
\underbrace{\|f^t-f\|_{L_q(\Omega)}}_{= o(1), Lemma~\ref{lemma:phit}, (ii)}. 
\een
Now using \eqref{eq:At_A}  and   \eqref{eq:f_ft} to estimate the right hand side of \eqref{eq:estimate_zt} yields
$\lim_{t\searrow 0}\|u^t-u\|_{W^1_q(\Omega)} = 0 $.  Since $q>2$ the space $W^1_{\Gamma,q}(\Omega)$ embeds 
continuously into $C_{\Gamma}(\Omega)$ and we obtain 
$\lim_{t\searrow 0}\|u^t-u\|_{C(\overbar \Omega)}=0$.  

Finally item $(b)$ follows since for $f\in W^1_q(D)$, 
$q>2$, we obtain the estimate 
$\|f^t-f\|_{L_q(D)} \le ct$ (cf. Lemma~\ref{lemma:phit}, (ii)). This finishes the proof.  
\end{proof}

\subsection{Analysis of the averaged adjoint state equation}
At first we introduce for fixed 
$y\in \overbar \Omega$ and $t\ge 0$ the Lagrangian function:
\ben
G_y(t,v,w) := \Psi(\Phi_t(y),v(y))  + \int_{\Omega} \beta(|B(t)\nabla v|^2)A(t)\nabla v\cdot \nabla w + \xi(t)vw  - f^t w\; dx, 
\een
where $v\in W^1_q(\Omega)$ and $w\in W^1_{q'}(\Omega)$ with $q>2$ and $q':= q/(q-1)$. 
Notice that $G_y=G_y^X$ also depends  on the vector field $X$, however, to keep the notation 
simple we omit this dependency. In the rest of the paper we assume $f\in W^1_q(D)$. 

\begin{definition}
	Let $y\in \overbar \Omega$ be fixed and $q\in (2,q_0]$, where $q_0$ is as in Lemma~\ref{lem:invertibility}.  We introduce the \emph{averaged 
		adjoint equation} as:
	\ben\label{eq:averaged_G}
	\text{ Find } p_{y}^t\in  W^1_{\Gamma,q'}(\Omega), \quad \int_0^1 d_v G_y(t,su^t+(1-s)u,p_{y}^t)(\varphi)\; ds =0 \quad \text{ for all } \varphi \in W^1_{\Gamma,q}(\Omega). 
	\een
	The function $p_{y}^t$ is referred to as \emph{averaged 
		adjoint state}.
\end{definition}

The reason for introducing the averaged adjoint equation is the following identity
\ben\label{eq:identity_important}
G_y(t,u^t,p^t_y)-G_y(t,u,p^t_y) = \int_0^1 d_v G_y(t,su^t+(1-s)u,p^t_y)(u^t-u)\; ds = 0,  
\een
where the last equality follows in view of \eqref{eq:averaged_G} and 
$u^t - u \in W^1_{\Gamma,q}(\Omega)$. Now with the Lagrangian $G_y$ the 
shape functions $J_\infty(\cdot)$ and $j(\cdot)$ 
can be expressed as
\ben\label{eq:J_pert_red}
J_\infty(\Omega_t^\Gamma) =  \max_{y\in \overbar 
	\Omega}j(\Omega^{\Gamma,y}_t), \quad j(\Omega_t^{\Gamma,y}) = 
G_{y}(t,u,p^t_y), \quad y\in  \overbar \Omega. 
 \een
Consequently it suffice to study the differentiability of $t\mapsto \max_{y\in \overbar \Omega}G_y(t,u,p^t_y)$ and 
$t\mapsto G_y^X(t,u,p_y^t)$ in order
to prove that $J_\infty(\cdot)$ is Eulerian semi-differentiable at $\Omega^\Gamma\in \Xi$
	and $j(\cdot)$ is shape differentiable at all $\Omega^{\Gamma,y}$, where $\Omega^\Gamma\in \Xi$ and 
	    $y\in \overbar \Omega$.
This is the content of the following two sections.  At first we 
study the averaged adjoint equation. We notice that \eqref{eq:averaged_G} is equivalent to
\ben\label{eq:averaged_adjoint_equation}
\int_\Omega b^t(x,u^t,u) \; L p_{y}^t \cdot 
L\varphi \; dx = -  \bar \Psi^t(y,u^t,u) \varphi (y) \quad\text{ for all } 
\varphi \in W^1_{\Gamma,q}(\Omega), 
\een
where
\ben
b^t(x,u^t,u) := \int_0^1 \partial_\zeta a^t(x,sLu^t(x)+(1-s)Lu(x)) \; ds, 	
\een
\ben
\bar \Psi^t(y,u^t(y),u(y)) := \int_0^1 \partial_{\zeta} \Psi(\Phi_t(y),su^t(y)+(1-s)u(y))\; ds. 	
\een
In view of 
\ben
\partial_\zeta a^t(x,\zeta) = \begin{pmatrix}
	 \xi(t,x) \zeta_0 \\ 
       	\beta(|B(t,x)\hat\zeta|^2)  A(t,x) + 2 
\beta'(|B(t,x)\hat\zeta|^2) \; A(t,x)\hat \zeta \otimes 
B(t,x)\hat\zeta  
\end{pmatrix}, \quad \zeta = \begin{pmatrix}
	\zeta_0 \\ \hat\zeta
\end{pmatrix},
\een
it immediately follows from Assumption~\ref{assmp:trans}, item 4, that  there is 
a constant $c>0$ so that $\|b^t(\cdot,u^t,u)\|_{L_\infty(\Omega)}\le c$ 
for all $t$. Notice 
that at $t=0$ equation \eqref{eq:averaged_adjoint_equation} reduces to the 
usual adjoint state equation:
\ben\label{eq:adjoint}
\text{ find } p_{y}\in W^1_{q'}(\Omega), \quad \int_{\Omega} b(x,u) Lp_{y} \cdot L\varphi \;dx 
= - \partial_u \Psi(y,u(y)) \varphi(y)\quad \text{ for all }\varphi \in  W^1_{\Gamma,q}(\Omega), 
\een
where $b(x,\cdot):=b^0(x,\cdot,\cdot )$.  
We associate with $b^t$ the (linear) operator
$
\mathcal B^t_{q'}:W^1_{\Gamma,q'}(\Omega) \rightarrow W^{-1}_{\Gamma,q'}(\Omega)$ defined by 
$\langle \mathcal \mathcal B^t_qv,w\rangle := \int_\Omega b^t(x,u^t,u) Lv \cdot Lw \, 
dx.$

The proof of the following lemma follows \cite{MR2823475}.
\begin{lemma}\label{lem:operator_Bt}
	Let $\Omega^\Gamma\in \Xi$ with associated $q_0>2$ be given. Then 
there is exists $\delta >0$, so that the averaged operator $\mathcal B^t_{q'}:W^1_{\Gamma,q'}(\Omega) \rightarrow W^{-1}_{\Gamma,q'}(\Omega) $
is an isomorphism for all $t \in [0,\tau]$.  Moreover, there is a constant $c>0$, so that for all 
$t\in [0,\delta]$,
\ben
\|(\mathcal B^t_{q'})^{-1}f- (\mathcal B^t_{q'})^{-1}g\|_{W^1_{\Gamma,q'}(\Omega)} \le c \|f-g\|_{W^{-1}_{\Gamma,q'}(\Omega)} \quad \text{ for all } f,g\in W^1_{\Gamma,q'}(\Omega).	
\een
\end{lemma}
\begin{proof}
Let $\eps >0$ be fixed. Using Assumption~\ref{assmp:trans} it is readily checked 
that there is $\delta >0$ so that for all $t\in [0,\delta]$ the function
 $b^t(x,\zeta):= b^t(x,u^t,u)\zeta $ satisfies \eqref{ass:b} with 
 $m=\min\{1,k\}-\eps$ and $M=\max\{K,1\}+\eps$ for $t$ sufficiently small. Hence there 
 is $\delta >0$ so that the mapping
 $ \mathcal B_q^t:W^1_{\Gamma,q}(\Omega) \rightarrow W^{-1}_{\Gamma,q}(\Omega)$ 
	is an isomorphism for all $t\in [0,\delta]$. Thus by the closed range theorem also the adjoint 
	$(\mathcal B_{q}^t)^*= \mathcal B_{q'}^t:W^1_{\Gamma,q'}(\Omega) \rightarrow 
	W^{-1}_{\Gamma,q'}(\Omega)$ is an isomorphism with continuous inverse and we 
	finish the proof. 
\end{proof}
\begin{lemma}\label{lem:averaged_weak}
	Let $\Omega^\Gamma\in \Xi$ with associated $q_0>2$ be given. Assume $y_t:\R \rightarrow \R^2$ is a 
	function that is continuous from the right in $t=0$ with $y(0)=y\in \overbar \Omega$. For $t\ge 
    0$ we denote by $p_{y_t}^t\in W^1_{\Gamma,q'}(\Omega)$ the solution of 
	\ben\label{eq:averaged_adjoint}
	\int_{\Omega} b^t(x,u^t,u)Lp_{y_t}^t \cdot L\varphi  
	\;dx =   - \bar \Psi^t(y_t,u^t(y_t),u(y_t))  \varphi (y_t) \quad 
	\text{ for all } \varphi\in W^1_{\Gamma,q}(\Omega),
	\een
     where $q>2$ is the conjugate of $q'$, that is, $1/{q'}+1/q=1$.
    Then $1<q'<2$ and  we get
    $
    p_{y_t}^t \rightharpoonup p_y$  weakly in $W^1_{\Gamma,q'}(\Omega),$
where $p_y$ denotes the solution of \eqref{eq:adjoint}. 
\end{lemma}
\begin{proof}
	By Sobolev's embedding the inclusion mapping 
	$ E_\Gamma: W^1_{\Gamma,q}(\Omega) \rightarrow C_\Gamma(\Omega)$ is 
	continuous for all $q>2$. Thus the adjoint 
	$E_\Gamma^*:C_\Gamma(\Omega) \rightarrow W^1_{\Gamma,q}(\Omega)$ is 
	continuous, too. As the mapping 
	$ \alpha^t \delta_{y_t}:C_\Gamma(\Omega)\rightarrow \R,\; f \mapsto \alpha^t f(y_t)  $
	, where  $ \alpha^t =: \bar \Psi^t(y_t,u^t(y_t),u(y_t)) \in \R$, is 
	continuous, we can rewrite  
	\eqref{eq:averaged_adjoint} as
	\ben
      	\langle \mathcal B_{q'}^tp_{y_t}^t, 
	\varphi\rangle_{W^1_{\Gamma,q'},W^{-1}_{\Gamma,q'}} =  -  \langle 
	E_{\Gamma}^* (\alpha^t 
	\delta_{y_t}),\varphi\rangle_{W^1_{\Gamma,q'},W^{-1}_{\Gamma,q'}} 
	\quad \text{ for all } \varphi\in W^1_{\Gamma,q}(\Omega).
	\een
         Now applying 
	Lemma~\ref{lem:averaged_weak} yields
	$
     \|p_{y_t}^t\|_{W^1_{q'}(\Omega)} \le c \|E_{\Gamma}^* (\alpha^t \delta_{y_t})\|_{W^{-1}_{\Gamma,q'}(\Omega)} \le c \alpha^t\| \delta_{y_t}\|_{(C_\Gamma(\Omega))^*}\le c	
     $	
     for all $t\in [0,\delta]$. 
        So for each real nullsequence $(t_n)$ there is a subsequence and $z\in W^1_{\Gamma,q'}(\Omega)$, still 
	indexed the same, such that  
	$p_{y_{t_n}} \rightharpoonup z$ in $W^1_{\Gamma,q'}(\Omega)$. Therefore 
	passing to the limit in \eqref{eq:averaged_adjoint}, we conclude by 
	uniqueness of the adjoint state equation that $z=p_y$. This also shows $p_{y_{t}} \rightharpoonup p_y$ in 
	$W^1_{\Gamma,q'}(\Omega)$ as $t\searrow 0$. 	
\end{proof}

\subsection{Shape derivative of $j(\cdot)$ via averaged adjoint}
 Let $X\in \ac C^1(D,\R^2)$ be a given vector field and $\Phi_t$ the 
corresponding flow. Let $\Omega^\Gamma\in \Xi$ and $y\in \overbar \Omega$. Then the perturbation of the set
$\Omega^{\Gamma,y}=(\Omega^\Gamma,y)$ is defined by $ \Omega_t^{\Gamma,y} := (\Omega_t,\Gamma_t,y_t)$, where
$\Omega_t := \Phi_t(\Omega)$, $\Gamma_t:=\Phi_t(\Gamma)$ and $y_t := \Phi_t(y)$. 
The Eulerian semi-derivative of $j(\cdot)$ at $\Omega^{\Gamma,y}$ in direction $X$ is then defined by
$
dj(\Omega^{\Gamma,y})(X) = \lim_{t\searrow 0}(j(\Omega^{\Gamma,y}_t) - j(\Omega^{\Gamma,y}))/t.
$
Let us now prove that $j(\cdot)$ is in fact shape differentiable.
\begin{theorem}\label{thm:shape_derivative_j}
	Let $\Omega^\Gamma\in \Xi$ and $y\in \overbar \Omega$ be given 
	and assume $2<q<q_0$.
   The shape function $j(\cdot)$ is 
    shape differentiable at every $ \Omega^{\Gamma,y}
    $  and the derivative in direction $X\in \ac C^1(D,\R^2)$ is given by
    \ben\label{eq:shape_derivative_j}
    dj(\Omega^{\Gamma,y})(X) = \partial_t G_y^X(0,u,p_y),
    \een
    where $(u,p_y)\in W^1_{\Gamma,q}(\Omega)\times  W^1_{\Gamma,q'}(\Omega)$ 
    solves \eqref{eq:state} and \eqref{eq:adjoint}, respectively.  
\end{theorem}
It is sufficient to prove the following lemma. 
\begin{lemma}\label{lem:averaged}
	Let $\Omega^\Gamma\in \Xi$ and $y\in \overbar \Omega$ be given 
	and assume $2<q<q_0$.
        For all functions $y_t=y(t):\R\rightarrow \R^2$ that are continuous from the 
        right in $t=0$, we have
	\ben\label{eq:lim_G_t}
      	\lim_{t\searrow 0} \frac{G_{y_t}(t,u^{t},p_{y_t}^t) - G_{y_t}(0,u,p_{y_t}^t)}{t} = \partial_t G_y(0,u,p_y), 
	\een
     where $(u,p_y)\in W^1_{\Gamma,q}(\Omega)\times  W^1_{\Gamma,q'}(\Omega)$ solves \eqref{eq:state} and \eqref{eq:adjoint}, respectively. 
\end{lemma}
\begin{proof}
	By definition of the function $p_{y_t}^t$, $t>0 $, we get (cf. \eqref{eq:identity_important})
    \ben
    \frac{G_{y_t}(t,u^{t},p_{y_t}^t) - G_{y_t}(0,u,p_{y_t})}{t} = \frac{G_{y_t}(t,u,p_{y_t}^t)-G_{y_t}(0,u,p_{y_t}^t)}{t}. 
    \een
We want to pass to the limit on the right hand side. To do so notice
\ben
\begin{split}
	\frac{\Psi(\Phi_t(y_t),u(y_t)) - \Psi(y_t,u(y_t))}{t} = \int_0^1 
	\nabla \Psi(s\Phi_t(y_t)+(1-s)y_t,u(y_t))\; ds \frac{\Phi_t(y_t)-y_t}{t} 
\end{split}
\een
and consequently
\ben\label{eqaveraged_1}
\begin{split}
	\left|\frac{\Psi(\Phi_t(y_t),u(y_t)) - \Psi(y_t,u(y_t))}{t} - \nabla_y \Psi(y,u(y)) \right| \le  c 
	\underbrace{\left\|\frac{\Phi_t-\text{id}}{t} - X\right\|_{C(\overbar 
			D,\R^{2,2})}}_{\rightarrow 0, \text{ in view of Lemma~\ref{lemma:phit}}} .
\end{split}
\een
By Lemma~\ref{lem:averaged_weak}, we obtain $p_{y_t}^t \rightarrow  p_y$ in  $W^1_{\Gamma,q'}(\Omega)$ for 
$q'=q/(q-1)$. Thus Lemma~\ref{lemma:phit} implies
\ben\label{eqaveraged_2}
\begin{split}
	& \int_{\Omega}\frac{\beta(|B(t)\nabla u|^2)A(t) -\beta(|\nabla u|^2)}{t}  \nabla p_{y_t}^t \cdot \nabla u + \frac{\xi(t)-1}{t} 
p_{y_t}^t u \; dx - \int_{\Omega} \frac{f^t-f}{t} p_{y_t}^t\; dx \\
 \rightarrow & \int_{\Omega} \beta(|\nabla u|^2) A'(0)\nabla p_y \cdot 
\nabla u 
+ 2\beta'(|\nabla u|^2)B'(0)\nabla u\cdot \nabla u \nabla p_y \cdot \nabla u \; dx+	\int_{\Omega} \divv(X) p_yu - 
  f' p_y\; dx    
\end{split}
\een
as $t\searrow 0$. 
Now \eqref{eqaveraged_1} and \eqref{eqaveraged_2} together imply 
\eqref{eq:lim_G_t} and thus our claim. 
\end{proof}
\begin{proof}[Proof of Theorem~\ref{thm:shape_derivative_j}]
According to \eqref{eq:J_pert_red} we have 
$j(\Omega_t^{\Gamma,y}) = G_y(t,u,p_{y}^t)$ for all $t$ and all $y\in \overbar \Omega$. 
     So an application of Lemma~\ref{lem:averaged} with $y(t)\equiv y$, yields
     $
     dj(\Omega^{\Gamma,y})(X) = \dt G_y(t,u^t,p_y)|_{t=0} = \partial_t G_y(0,u,p_y).
     $
\end{proof}
Now we can present explicit formulas for the shape derivative of 
$j(\cdot)$. 
\begin{corollary}\label{cor:tensor_form_j}
    \begin{itemize}
		    Let $\Omega^\Gamma\in \Xi$ and assume $2<q <q_0$.  
	  \item[(a)]  
  The shape derivative of $j(\cdot)$  at $ \Omega^{\Gamma,y}
    $, $y\in \overbar \Omega$, in direction $X\in \ac C^1(D,\R^2)$ is given by
    \ben\label{eq:shape_tensor_j}
    dj(\Omega^{\Gamma,y})(X) = \int_{\Omega} \Sb_1(u,p_y):\partial X + 
    \Sb_0(u,p_y)\cdot X \;dx  +  X(y) \cdot \nabla_y \Psi(y,u(y)),	
    \een
    where
    \ben\label{eq:formula_S1}
    \begin{split}
    \Sb_1(u,p_y) :=&  (\beta(|\nabla u|^2)\nabla u\cdot \nabla p_y + up_y- fp_y)  I - \beta(|\nabla u|^2) (\nabla u\otimes \nabla p_y + \nabla p_y \otimes \nabla u)\\
                       & - 2\beta'(|\nabla u|^2)(\nabla u\cdot \nabla p_y) \nabla u\otimes \nabla u 
    \end{split}
    \een
    \ben\label{eq:formula_S0}
    \Sb_0(u,p_y) := - \nabla f p_y. 	
    \een
Here, $(u,p_y)\in W^1_{\Gamma,q}(\Omega)\times  W^1_{\Gamma,q'}(\Omega)$ solve \eqref{eq:state} and \eqref{eq:adjoint}, respectively.  
    \item[(b)]
    Assume $\partial \Omega\in C^1$, $u\in H^2(\Omega)$, $f\in H^2(\Omega)$ and 
    $p_y\in H^2(\Omega\setminus \{y\})$ for 
    $y\in \Omega$ and 
    $p_y\in H^2(\Omega)$ for $y\in \partial \Omega$. Then for every 
    $y\in \Omega$,   
    \ben\label{eq:divv_S1_S02}
    -\divv(\Sb_1(u,p_y)) + \Sb_0(u,p_y) =0 \quad \text{ a.e. in } \Omega \setminus \{y\}
    \een
    and for every $y\in \partial \Omega$,
    \ben\label{eq:divv_S1_S02-2}
    -\divv(\Sb_1(u,p_y)) + \Sb_0(u,p_y) =0 \quad \text{ a.e. in } \Omega. 
    \een
    Moreover, for all $y\in \Omega$,
    \ben\label{eq:shape_tensor_j_boundary}
    dj(\Omega^{\Gamma,y})(X) = \int_{\partial \Omega} \Sb_1(u,p_y)\nu\cdot 
    \nu (X\cdot \nu)  \;ds  + (\Sb_1(u,p_y)\nu \otimes \delta_y) X  +  X(y) \cdot \nabla_y \Psi(y,u(y)),	
    \een
    where
    $
    (\Sb_1(u,p_y)\nu \otimes \delta_y) X := \lim_{\delta \searrow 0} \int_{\partial B_\delta(y)} \Sb_1(u,p_y)\nu \cdot X \; ds
    $
    and for all $y\in \partial \Omega$,
    \ben\label{eq:shape_tensor_j_boundary-2}
    dj(\Omega^{\Gamma,y})(X) = \int_{\partial \Omega} \Sb_1(u,p_y)\nu\cdot 
    \nu (X\cdot \nu)  \;ds + X(y) \cdot \nabla_y \Psi(y,u(y)). 	
     \een
Here $B_\delta(y)$ denotes the ball centered at $y$ with radius $\delta$.  
        \end{itemize}
\end{corollary}
\begin{proof}
    At first by Theorem~\ref{thm:shape_derivative_j},  
    $j(\Omega^{\Gamma,y})(X) = \partial_t G_y^X(0,u,p_y)$ and
    \ben\label{eq:partial_G_y}
    \begin{split}
    \partial_t G_y^X(0,u,p_y)  = & \int_{\Omega} \beta(|\nabla u|^2) A'(0)\nabla p_y \cdot 
\nabla u 
+ 2\beta'(|\nabla u|^2)B'(0)\nabla u\cdot \nabla u \nabla p_y \cdot \nabla u \; dx\\
       &+	\int_{\Omega} \divv(X) p_yu - 
  f' p_y\; dx - \int_{\Omega} f' p_y 
    \; dx +  X(y) \cdot \nabla_y \Psi(y,u(y)),  
    \end{split}
    \een
    for all $X\in \ac C^1(D,\R^2)$,
where according to Lemma~\ref{lemma:phit},
    $  A'(0) = \divv(X)I - \partial X - \partial X^\top $
    and $ f' := \divv(X)f + \nabla f \cdot X$. Therefore
    it is readily verified that \eqref{eq:partial_G_y} can be brought into the tensor form 
    \eqref{eq:shape_tensor_j}.
    
    Let us now prove that \eqref{eq:shape_tensor_j} is in fact equivalent to 
    \eqref{eq:shape_tensor_j_boundary} when $u\in H^2(\Omega)$, 
    $f\in H^2(\Omega)$ and $p_y\in H^2(\Omega)$ for
    $y\in\partial \Omega$ and $p_y\in H^2(\Omega\setminus \{y\})$ for $y\in \Omega$. 
    Let $y\in \Omega$ be given and choose $\delta>0$ such that $\overline{B_\delta(y)}\subset \Omega$ and 
    define the Lipschitz domain $\Omega_\delta := \Omega \setminus \overbar{B_\delta(y)}$. 
    By Nagumo's theorem  it follows that for all $\delta> 0$ so that 
    $\overbar{B_\delta(y)}\subset \Omega$, $dj(\Omega^{\Gamma,y})(X)=0$ for all 
    $X\in C^1_c(\Omega_\delta, \R^2)$.  But according to 
    \eqref{eq:shape_tensor_j} this is equivalent to
    \ben\label{eq:B_delta}
    \int_{\Omega} \Sb_1(u,p_y):\partial X + \Sb_0(u,p_y)\cdot X \;dx = 0 \quad \forall X\in C^1_c(\Omega_\delta,\R^2).
    \een
    Now partial integration and the fundamental theorem of the calculus of variations 
    yield for all small $\delta >0$
    $    -\divv(\Sb_1(u,p_y)) + \Sb_0(u,p_y) =0$ a.e. on $\Omega_\delta
    $
    and hence
    \ben\label{eq:divv_S1_S0}
    -\divv(\Sb_1(u,p_y)) + \Sb_0(u,p_y) =0 \quad \text{ a.e. on } 
    \Omega\setminus \{y\}.
    \een

    Setting $\Sb_1 := \Sb_1(u,p_y)$ and $\Sb_0 := \Sb_0(u,p_y)$,
    we obtain
    \ben\label{eq:dj_Omega}
    \begin{split}	
        dj(\Omega^{\Gamma,y})(X)  = & \int_{B_\delta(y)}\Sb_1:\partial X + \Sb_0 \cdot 
	X\; dx + \int_{\Omega_\delta} (\underbrace{-\divv(\Sb_1) + \Sb_0}_{=0, 
		\eqref{eq:divv_S1_S0}})\cdot X\; dx + \int_{\partial 
            \Omega} \Sb_1\nu \cdot X\; ds \\
    &+ X(y) \cdot \nabla_y \Psi(y,u(y)) + \int_{\partial 
	    B_\delta(y)} \Sb_1\nu \cdot X\; ds\quad \text{ for }X\in C^1_c(D,\R^2).
    \end{split}
    \een 
    Further H\"older's inequality shows 
     \begin{align*} 
	     \left|\int_{B_\delta(y)}\Sb_1:\partial X + \Sb_0 \cdot X\; dx\right| \le 
    |B_\delta(y)|^{1/{q'}}\|X\|_{C^1}(\|\Sb_1\|_{L_q(\Omega,\R^{2,2})} + \|\Sb_0\|_{L_q(\Omega,\R^2)}) 
    \end{align*}	
    and the right hand side goes to zero as $\delta \searrow 0$.
    Consequently
    \ben\label{eq:lim_delta_Sb_1}
    \lim_{\delta \searrow 0}\int_{\partial 
        B_\delta(y)} \Sb_1\nu \cdot X\; ds  = dj(\Omega^{\Gamma,y})(X)-\int_{\partial 
        \Omega} \Sb_1\nu \cdot X\; ds + X(y) \cdot \nabla_y \Psi(y,u(y))   
    \een
for all $X\in \ac C^1(D,\R^2)$. Observe
    $
    X\mapsto  (\Sb_1(u,p_y)\nu \otimes \delta_y) X :=   \lim_{\delta \searrow 0}\int_{\partial 
        B_\delta(y)} \Sb_1\nu \cdot X\; ds$ is linear and continuous as a mapping 
    $\ac C^1(D,\R^2) \rightarrow \R$.
    Define $X_\tau := X-(X\cdot \tilde \nu)\tilde \nu$,  where 
    $\tilde \nu$ is a smooth extension of $\nu$ such that $\supp\tilde \nu\subset D \setminus\overbar B_\delta(y)$ and $X\in C^1_c(D,\R^2)$. Then $dj(\Omega^{\Gamma,y})(X_\tau)=0$ which is equivalent to
    \ben\label{eq:Sb_1_boundary}
    \int_{\partial 
        \Omega} \Sb_1\nu\cdot X\; ds = \int_{\partial 
	\Omega} \Sb_1\nu\cdot \nu (X\cdot \nu)\; ds \quad \text{ for all } X\in C^1_c(D,\R^2) . 
\een
So inserting \eqref{eq:Sb_1_boundary} into \eqref{eq:dj_Omega}, we recover  \eqref{eq:shape_tensor_j_boundary}.

Now let $y\in \partial \Omega$. Then $dj(\Omega^{\Gamma,y})(X)=0$ for all 
    $X\in C^1_c(\Omega, \R^2)$ and this is equivalent to
    $
    \int_{\Omega} \Sb_1(u,p_y):\partial X + \Sb_0(u,p_y)\cdot X \;dx = 0     
    $ for all $X\in C^1_c(\Omega,\R^2)$,  
   from which we conclude by partial integration and the fundamental theorem of calculus 
   of variations,
   $
   -\divv(\Sb_1(u,p_y)) + \Sb_0(u,p_y) =0$ 
   a.e. on $\Omega.$
   Hence integrating by parts we obtain 
   \ben
    \begin{split}	
        dj(\Omega^{\Gamma,y})(X)  = \int_{\partial \Omega} \Sb_1\nu \cdot X\; ds + X(y) \cdot 
       \nabla_y \Psi(y,u(y)) \quad \text{ for }X\in \ac C^1(D,\R^2)
    \end{split}
    \een 
 and since also in this case \eqref{eq:Sb_1_boundary} is valid we get \eqref{eq:shape_tensor_j_boundary-2}.
\end{proof}

\begin{remark}
	Notice that if we strengthen the assumption in item (b) of the previous theorem
	and assume for all $y\in \Omega$, $p_y\in H^2(\Omega)$, then it 
	follows from \eqref{eq:lim_delta_Sb_1}  by partial integration
$(\Sb_1(u,p_y)\nu \otimes \delta_y) X = 0.$
\end{remark}

\subsection{Eulerian semi-derivative of $J_\infty(\cdot)$}
The following theorem is a Danksin type theorem and follows essentially from 
the proof of \cite[Theorem 2.1, p.524]{MR2731611}. Since our setting is different
from the one in the book we give a proof. 
\begin{lemma}\label{lem:danskin}
Let $K\subset \R^d$ be a compact set, $\tau>0$ a positive number and 
$g:[0,\tau]\times K\rightarrow \R$ some function. Define for 
$t\in [0,\tau]$ the set
$
R^t = \{z\in K:\; \max_{x\in K} g(t,x) = g(t,z) \}
$
with the convention $R:=R^0$. 
 Assume that
\begin{itemize}
	\item[(A1)] for all $ x\in R$, the partial derivative $\partial_t 
		g(0^+,x)$ exists, 
	\item[(A2)] for all $t\in[0,\tau]$, the function $x\mapsto g(t,x)$ is upper semi-continuous,
	\item[(A3)] for all real nullsequences $(t_n)$, $t_n\searrow 0$, and 
		all sequences $(y_{t_n})$, $y_{t_n}\in R^{t_n}$ converging 
		to some $y\in R$, we have
	       \ben
	       \lim_{n\rightarrow \infty} \frac{g(t_n,y_{t_n}) - 
		       g(0,y_{t_n})}{t_n} = \partial_t g(0^+,y). 
	       \een	       
\end{itemize}
Then
\ben
\frac{d}{dt}\left( \max_{x\in K} g(t,x) \right)_{t=0} = \max_{x\in R} \partial_t g(0^+,x).	
\een 
\end{lemma}
\begin{proof}
	Due to assumption (A2) and the compactness of $K$ the set $R^t$ is nonempty for all 
	$t\in [0,\tau]$. Furthermore, by definition, for all $t\ge 0$, $y_t\in 
	R^t$, and $y\in R$ we have
        $g(t,y_t)  \ge g(t,y)$ and $g(0,y)  \ge g(0,y_t).$
        Using these two inequalities we obtain 
        $
		g(t,y_t) - g(0,y)  \ge g(t,y) - g(0,y)
        $
        and also 
        $
		g(t,y_t)- g(0,y)  \le  g(t,y_t) - g(0,y_t)
        $ and consequently
	\ben\label{eq:ineq_chain}
	 g(t,y) - g(0,y)  \le  g(t,y_t) - g(0,y) \le g(t,y_t) - g(0,y_t).
	\een
        Setting $\delta(t):=(g(t,y_t) - g(0,y) )/t $ it is 
        sufficient to show that 
        $\liminf_{t\searrow 0}\delta(t)=\limsup_{t\searrow 0}\delta(t)$ and 
	one of the limits is finite. 
	By assumption (A1) and \eqref{eq:ineq_chain}, we obtain the chain of inequalities
	$
	\partial_t g(0^+,y) \le  \liminf_{t\searrow 0} \delta(t) \le \limsup_{t\searrow 0} \delta(t)
	$ 
for all $y\in R$. 
        Since the previous inequality is true for all $y\in R$ it implies
        \ben\label{eq:le_max}
         \max_{y\in R}\partial_t g(0^+,y) \le  \liminf_{t\searrow 0} \delta(t) \le \limsup_{t\searrow 0} \delta(t).
         \een
         Now $K$ is compact and $y_t\in K$ for all 
	 $t\ge 0$, so we find for each nullsequence $(t_n)$ a subsequence, still indexed the 
         same, and $y\in K$, such that $y_{t_n}\rightarrow y$ as 
	 $n\rightarrow \infty$. We need to show that $y\in R$. In fact it follows for all $x\in K$,
     $
     g(t_n,x) \le g(t_n,y_{t_n}) = g(0,y_{t_n}) +  t_n \frac{g(t_n,y_{t_n})-g(0,y_{t_n})}{t_n}
     $
     and thus using Assumption (A2) we get for all $x\in K$,
     \ben
     g(0,x) = \limsup_{n\rightarrow \infty} g(t_n,x) \le  \limsup_{n\rightarrow \infty} g(0,y_{t_n}) +  \limsup_{n\rightarrow \infty} t_n \frac{g(t_n,y_{t_n})-g(0,y_{t_n})}{t_n} \le g(0,y).
     \een
     This shows that $y$ is a maximum of $g(0,\cdot)$, that is, $y\in R$.
     We deduce from \eqref{eq:ineq_chain} and Assumption (A3),
             $
	     \limsup_{t\searrow 0}\delta(t) \le \lim_{n\rightarrow \infty}\frac{g(t_n,y_{t_n}) - g(0,y_{t_n})}{t_n} = \partial_t g(0^+,y) 
	 $
         and hence
        $
        \limsup_{t\searrow 0}\delta(t) \le \partial_t g(0^+,y) \le \max_{y\in R}\partial_t g(0^+,y).   
        $
        Finally combining the previous inequality with \eqref{eq:le_max} yields,\newline
	$
       	\max_{y\in R} \partial_t g(0^+,y) \le  \liminf_{t\searrow 0} 
	\delta(t) \le \limsup_{t\searrow 0} \delta(t) \le  \max_{y\in R} \partial_t g(0^+,y)
	$
         and thus the desired result.
    \end{proof}

Let us define the set
\ben\label{eq:R}
R(\Omega^\Gamma) := \left\{x\in \overbar \Omega: J_\infty(\Omega^\Gamma)  = \Psi(x,u(x))   \right\}.
\een
We can now prove the following main result. 
\begin{theorem}\label{thm:main_eulerian_semi}
	Let $\Omega^\Gamma\in \Xi$ be given and suppose $q\in (2,q_0]$. Then the Eulerian semi-derivative of the 
    shape function
    $J_\infty$ given by \eqref{eq:cost_max} at $\Omega^\Gamma$ in direction 
    $X\in \ac C^1(D,\R^2)$ is given by 
    \ben
    dJ_\infty(\Omega^\Gamma)(X) = \max_{y\in R(\Omega^\Gamma)} \partial_t G_{y}^X(0,u,p_y),   
    \een
where $(u,p_y)\in W^1_{\Gamma,q}(\Omega)\times  W^1_{\Gamma,q'}(\Omega)$ solve \eqref{eq:state} and \eqref{eq:adjoint}, respectively.  
\end{theorem}
\begin{proof}
We apply Lemma~\ref{lem:danskin} with 
$g(t,y):=G_{y}(t,u^t,p_y)=G_y(t,u,p_{y}^t)$ and  $K:= \overbar \Omega $. 
Assumption (A1) is clear. Assumptions (A2) follows from 
Lemma~\ref{lem:averaged_weak} and the continuity of $u$. Assumption (A3) is a 
consequence of
Lemma~\ref{lem:averaged}. 
Thus all assumptions are satsified and the claim follows. 
\end{proof}

The next theorem gives a complete characterisation of the Eulerian 
semi-derivative of $J_\infty(\cdot)$. We show that the Eulerian 
semi-derivative is related to the maximum of a boundary integral 
provided the state and adjoint state are more result. This can be
seen as a generalisation of \cite[Proposition 3.2]{lauraindistributed}.

\begin{theorem}\label{thm:eulerian_tensor}
	Suppose $\Omega^\Gamma\in \Xi$ and $ 2<q< q_0$. 
	\begin{itemize}
		\item[(a)]
	The Eulerian semi-derivative of the maximum function \eqref{eq:cost_max} 
	at $\Omega^\Gamma$ in direction $X\in \ac C^1(D,\R^2)$ is given by
\ben\label{eq:SD_adjoint}
 dJ_\infty(\Omega^\Gamma)(X) = \max_{y\in R(\Omega^\Gamma)}\left(\int_{\Omega} 
	 \Sb_1(u,p_y):\partial X + \Sb_0(u,p_y)\cdot X \;dx  + X(y) \cdot \nabla_y \Psi(y,u(y)) \right),   
 \een
where $\Sb_1,\Sb_0$ are defined in 
\eqref{eq:formula_S0},\eqref{eq:formula_S1}. The adjoint state $p_y\in W^1_{q'}(\Omega)$ 
solves for $y\in R(\Omega^\Gamma)$,
\ben
\int_{\Omega}b(x,u)L p_y \cdot L\varphi  \;dx = - \partial_u\Psi(y,u(y)) \varphi (y) \quad \text{ for all } \varphi\in W^1_{\Gamma,q}(\Omega), 
\een
and the state $u\in W^1_{\Gamma,q}(\Omega)$ solves the state equation \eqref{eq:state}. 
Moreover,
\ben\label{eq:inquality_S_1_S_0}
\int_{\Omega} \Sb_1(u,p_y):\partial X + \Sb_0(u,p_y)\cdot X \;dx  + X(y) \cdot 
\nabla_y \Psi(y,u(y)) \le 0 
\een
for all $X\in \ac C^1(\Omega,\R^2)$ and for all $y\in R(\Omega^\Gamma)$. 
\item[(b)]
	When $\partial \Omega\in C^1$, $u\in H^2(\Omega)$, $p_y\in H^2(\Omega\setminus \{y\})
	$ for all $y\in \Omega \cap R(\Omega^\Gamma)$ and $p_y\in H^2(\Omega)
	$ for all $y\in \partial \Omega \cap R(\Omega^\Gamma)$, 
then \eqref{eq:SD_adjoint} is equivalent to
\ben\label{eq:boundary_dJ_infty}
\begin{split} 
  dJ_\infty(\Omega^\Gamma)(X)  = \max_{y\in R(\Omega^\Gamma)}\bigg( \int_{\partial 
	  \Omega} \Sb_1(u,p_y)&\nu\cdot 
  \nu X_\nu  \;ds +  \chi_{\partial \Omega}(y)X_\nu(y) \nu(y)\cdot 
\nabla_y \Psi(y,u(y))\bigg)
\end{split}
\een
where $X\in \ac C^1(D,\R^2)$ and $X_\nu := X\cdot \nu$. Here $\chi_{\partial \Omega}$ denotes the 
characteristic function associated with $\partial \Omega$. 
\end{itemize}
\end{theorem}
\begin{proof}
Equation \eqref{eq:SD_adjoint} follows by 
combining Corollary~\ref{cor:tensor_form_j},
Theorem~\ref{thm:main_eulerian_semi} and Theorem~\ref{thm:shape_derivative_j}.

We now prove \eqref{eq:inquality_S_1_S_0}. By Nagumo's theorem it follows
    $dJ_\infty(\Omega^\Gamma)(X)=0$ for all $X\in \ac C^1(\Omega,\R^2)$ and this 
    implies,
    \ben\label{eq:dj_le_dJ}
    dj(\Omega^{\Gamma,y})(X)   \le dJ_\infty(\Omega^\Gamma)(X)=0	
    \een
for all $y\in R(\Omega^\Gamma)$ and all $X\in \ac C^1(\Omega,\R^2)$. 
   Taking into account \eqref{eq:shape_tensor_j} we recover  \eqref{eq:inquality_S_1_S_0}.
   
   Now under the assumption of item (b) we know from 
   Corollary~\ref{cor:tensor_form_j} that $dj(\Omega^{\Gamma,y})$ has the 
   form \eqref{eq:shape_tensor_j_boundary} for $y\in R(\Omega^\Gamma)\cap\Omega$. So inserting \eqref{eq:shape_tensor_j_boundary} into \eqref{eq:dj_le_dJ} and taking into account $X=0$ on 
   $\partial \Omega$, we obtain
   $
   (\Sb_1(u,p_y)\nu \otimes \delta_y) X  +  X(y) \cdot \nabla_y \Psi(y,u(y)) 
   \le 0$ for all $y\in \Omega \cap R(\Omega^\Gamma)$ and all $X\in \ac 
   C^1(\Omega,\R^2)$. 
      Since this inequality is true for all $X\in \ac C^1(\Omega,\R^2)$ we obtain 
  $ (\Sb_1(u,p_y)\nu \otimes \delta_y)  +  \nabla_y \Psi(y,u(y)) = 0$
  for all $y\in \Omega\cap  R(\Omega^\Gamma)$. Further 
  we get
   $
    X(y) \cdot \nabla_y \Psi(y,u(y)) \le 0$ for all $y\in \partial\Omega \cap 
    R(\Omega^\Gamma)$ and all $ X\in 
    C^1(\Omega,\R^2)$ with $X\cdot \nu=0$  on $ \partial 
    \Omega. 
   $
  Consequently \eqref{eq:boundary_dJ_infty} follows from \eqref{eq:shape_tensor_j_boundary} and \eqref{eq:shape_tensor_j_boundary-2}.
\end{proof}

\begin{corollary}\label{cor:boundary}
    Let $\Omega^\Gamma\in \Xi$.  Assume that $
    R(\Omega^\Gamma)\subset \partial \Omega$ and $\Gamma=\emptyset$. Then 
    \ben\label{eq:SD_boundary}
    dJ_\infty(\Omega^\Gamma)(X) = \max_{y\in R(\Omega^\Gamma)} X(y) \cdot \nabla_y 
    \Psi(y,u(y)) \quad \text{ for all } X\in \ac C^1(D,\R^2).
    \een 
\end{corollary}
\begin{proof}
This follows immediately from \eqref{eq:SD_adjoint}, since $p_y=0$ for all $y\in \Gamma_0=\partial \Omega$. 
\end{proof}

\begin{corollary}
    Let $\Omega^\Gamma\in \Xi$. If $
    R(\Omega^\Gamma)=\{y_0\}$ is a single-tone, then $J_\infty(\cdot)$ is shape 
    differentiable at $\Omega^\Gamma$. 
\end{corollary}

\begin{corollary}\label{cor:subadditive}
	Let $\Omega^\Gamma\in \Xi$. We have
	$
       	dJ_\infty(\Omega^\Gamma)(X+Y) \le dJ_\infty(\Omega^\Gamma)(X) +dJ_\infty(\Omega^\Gamma)(Y)	
	$
	for all $X,Y\in \ac C^1(D,\R^2)$. 
\end{corollary}

\begin{remark}
    We note that to show the differentiability of $J_\infty(\cdot)$ one might want to use the material derivative approach; cf. \cite{SokZol89}. 
    In our general setting this approach is difficult to apply as one would have to show the strong differentiability of 
    $t\mapsto u^t$ from $[0,\tau]$ into $W^1_{\Gamma,p}(\Omega)$ for some $q>2$. The weak differentiablity is not sufficient as 
    $W^1_{\Gamma,p}(\Omega)$ does not embed compactly into $C_\Gamma(\Omega)$.
\end{remark}

\section{Characterisation of stationary points of $J_\infty(\cdot)$}\label{sec:charac_stationary_points}
 This section is devoted to the characterisation of stationary points
 of the shape function $J_\infty(\cdot)$. We closely follow the approach of 
 \cite{MR1088479}, where finite dimensional problems are studied.
 Accordingly many results have to be carefully modified to account for the 
 infinite dimensionality of our problem. Throughout this section we suppose that 
 the assumptions of Theorem~\ref{thm:main_eulerian_semi}, item (a), are satisfied. 
 \subsection{Gradient of $j(\cdot)$}
 Let $\mathcal H(\Omega,\R^2)$ be some Hilbert space of functions from $\Omega$ into $\R^2$. According to  
 Corollary~\ref{cor:tensor_form_j} the shape derivative of $j(\cdot)$ at $\Omega^{\Gamma,y}$, 
 $y\in \overbar \Omega$ is given by
 \ben
dj(\Omega^{\Gamma,y})(X) = \int_{\Omega} \Sb_1(u,p_y):\partial X + \Sb_0(u,p_y)\cdot X 
\;dx  + X(y)\cdot \nabla_y\Psi(y,u(y))
\een
for all $X\in \ac C^1(D,\R^2)$.  
Assume that $\mathcal H(\Omega,\R^2)\subset \ac C^{0,1}(D,\R^2)$, such that
$
a\otimes \delta_y:\mathcal H(\Omega,\R^2)\rightarrow \R: X\mapsto a\cdot X(y)
$ is continuous for all $a\in \R^2$. 
Denoting by $\mathcal R:\mathcal H(\Omega,\R^2)\rightarrow (\mathcal H(\Omega,\R^2))^*$ 
the Riesz isomorphism then we have 
$\mathcal R^{-1}(dj(\Omega^{\Gamma,y})) = \nabla j(\Omega^{\Gamma,y})$ and by definition it 
satisfies for fixed $y\in \overbar \Omega$ the variational equation,
\ben\label{eq:X_y}
(\nabla j(\Omega^{\Gamma,y}),\varphi)_{\mathcal H} = 
dj(\Omega^{\Gamma,y})(\varphi) \quad \text{ for all } \varphi \in  \mathcal H(\Omega,\R^2).
\een
As a consequence \eqref{eq:SD_adjoint} can be written as
$
dJ_\infty(\Omega^\Gamma)(X) = \max_{y\in R(\Omega^\Gamma)} (\nabla j(\Omega^{\Gamma,y}),X)_{\mathcal H}.  
$
One way to construct the space $\mathcal H(\Omega,\R^2)$ is 
to define it as reproducing kernel Hilbert space associated with matrix-valued kernels of 
the form $K(x,y) = \phi(|x-y|^2/\sigma)I$, $\sigma> 0$, where
$\phi\in C^1(\R)$ is some smooth function. 
Then  \cite[Lemma 3.13]{eigelsturm16} provides an explicit 
formula for the gradient of $j(\Omega^{\Gamma,y})$.  An alternative way, also described in \cite{eigelsturm16}, is to choose $\mathcal H(\Omega,\R^2)$ as a finite element
space $V_N(\Omega,\R^2)$. This is described in more detail in the last 
section of this paper. In the following we fix the space 
$\mathcal H(\Omega,\R^2)$ and denote the gradient of $j(\cdot) $ simply by 
$\nabla j(\Omega^{\Gamma,y})$ always keeping in mind that it depends on the choice
of the space $\mathcal H(\Omega,\R^2)$ and the inner product chosen.

\subsection{Stationary points}
The following presentation is based on \cite[Chapter 3]{MR1088479}. We point 
out that there only the finite dimensional case was studied and we have to 
adapt our results to the infinite dimensional setting.

Let us begin with the definition stationary points. 
\begin{definition}
	The set $\Omega^\Gamma\in \Xi$ is said to be a stationary point for $J_\infty(\cdot)$ 
	with respect to perturbations in $\mathcal 
	H(\Omega,\R^2)$ if 
$ dJ_\infty(\Omega^\Gamma)(X)\ge 0$ for all  $X\in \mathcal H(\Omega,\R^2).  $
\end{definition}

Define the sets
\ben
H(\Omega^\Gamma) := \{ \nabla j(\Omega^{\Gamma,y}):\; y\in R(\Omega^\Gamma) \}
\een
and the convex hull of $H(\Omega^\Gamma)$ by 
\ben
L(\Omega^\Gamma) := \left\{\sum_{k=1}^n\alpha_k X^k: \; n\in \N, \; X^k\in H(\Omega^\Gamma),\;  \alpha_k\ge 0, \;  k=1,\ldots, n,\quad \sum_{k=1}^n\alpha_k =1  \right\} .  
\een
The closure of $L(\Omega^\Gamma)$ in $\mathcal H(\Omega,\R^2)$ is denoted by 
$\bar L(\Omega^\Gamma)$. 
We now show that  $H(\Omega^\Gamma)$ is closed and bounded.

\begin{remark}
	Notice that the set $\bar L(\Omega^\Gamma)$ is related to the Clarke 
	subdifferential; cf. \cite{MR1058436}.
\end{remark}
\begin{lemma}
	The set $H(\Omega^\Gamma)$ is closed and bounded in $\mathcal H(\Omega,\R^2)$. 
\end{lemma}
\begin{proof}
	We first show that $H(\Omega^\Gamma)$ is closed. Set 
	$X^y:=\nabla j(\Omega^{\Gamma,y})$. 
	Let $\{y_n\}$ be a sequence in $R(\Omega^\Gamma)$ such that $X^{y_n}\rightarrow 
	X$ in $\mathcal H(\Omega,\R^2)$. Since $R(\Omega^\Gamma)$ is compact there is a subsequence
	$\{y_{n_k}\}$ such that $y_{n_k}\rightarrow y\in R(\Omega^\Gamma)$ as 
	$k\rightarrow \infty$.  Hence Lemma~\ref{lem:averaged_weak} implies 
	$p_{n_k}\rightharpoonup p_y$ weakly in $W^1_{\Gamma,q'}(\Omega)$ and 
	 using \eqref{eq:X_y}, we get
\ben
\begin{split}
	(X^{y_{n_k}}, \varphi)_{\mathcal H} & = 	\left( \int_{\Omega} 
		\Sb_1(u,p_{y_{n_k}}):\partial \varphi + 
		\Sb_0(u,p_{y_{n_k}})\cdot \varphi \;dx  + \nabla_y\Psi(y_{n_k},u(y_{n_k}))\cdot 
		\varphi(y_{n_k}) \right)\\
	& \rightarrow \left( \int_{\Omega} \Sb_1(y,p_y):\partial \varphi + 
		\Sb_0(u,p_y)\cdot \varphi \;dx  + \nabla_y \Psi(y,u(y)) \right) \\
	& = (X, \varphi)_{\mathcal H}\quad \text{ for all }\varphi\in \mathcal H(\Omega,\R^2).
\end{split} 
\een 
Now as the weak limit and the strong limit coincide it 
follows $X=X^y$.
The boundedness of $H(\Omega^\Gamma)$ is obvious since 
$\overbar \Omega \rightarrow \R:\, y\mapsto \|X^y\|_{\mathcal H} $ is 
continuous and $\overbar \Omega$ compact. 
\end{proof}
With the definition of $H(\Omega^\Gamma)$ we can write the Eulerian semi-derivative 
of $J_\infty(\cdot)$ as\newline
$
dJ_\infty(\Omega^\Gamma)(X) = \max_{Z\in H(\Omega^\Gamma)}(Z,X)_{\mathcal H}	 
$
and the right hand side can be further rewritten. 
\begin{lemma}
	There holds for all $X\in \mathcal H(\Omega,\R^2)$,
	\ben
       	\max_{Z\in H(\Omega^\Gamma)}(Z,X)_{\mathcal H} = \max_{Z\in \bar L(\Omega^\Gamma)}(Z,X)_{\mathcal H}.
\een
\end{lemma}
\begin{proof}
Since $H(\Omega^\Gamma)\subset \bar L(\Omega^\Gamma)$, we immediately get the inequality
\ben\label{eq:H_le}
\max_{Z\in H(\Omega^\Gamma)}(Z,X)_{\mathcal H} \le \max_{Z\in \bar 
	L(\Omega^\Gamma)}(Z,X)_{\mathcal H}. 
\een
To show the other inquality let $\hat Z\in \bar L(\Omega^\Gamma)$. Then by definition we 
find a sequence $\{Z_n\}$ in $\mathcal H(\Omega,\R^2)$ such that 
$Z_n\rightarrow \hat Z$ in $\mathcal H(\Omega,\R^2)$ and
$
Z_n=\sum_{i=1}^n \alpha^i_n Z^{y^n_i} $ with $\sum_{i=1}^n\alpha^i_n =1,$ $\alpha^i_n\ge 0.
$
We obtain for all $n\ge 1$,
\ben
\begin{split}
	(Z_n, \varphi)_{\mathcal H} & = \sum_{i=1}^n \alpha^i_n (Z^{y^n_i},  
	\varphi)_{\mathcal H}
       	 \le \max_{Z\in H(\Omega^\Gamma)}(Z,\varphi)_{\mathcal 
		 H} \sum_{i=1}^n \alpha^i_n
                               = 	\max_{Z\in H(\Omega^\Gamma)}(Z,\varphi)_{\mathcal 
		H}.
\end{split}
\een
Passing to the limit $n\rightarrow \infty$ shows
$
(\hat Z,\varphi)_{\mathcal H} \le \max_{Z\in H(\Omega^\Gamma)}(Z,\varphi)_{\mathcal 
		H}
$
for all $\hat Z\in \bar L(\Omega^\Gamma)$. 
Taking the supremum over $\hat Z$ and taking into account inequality 
\eqref{eq:H_le} finishes the proof. 
\end{proof}

\begin{lemma}
	The set $\Omega^\Gamma\in \Xi$ is a stationary point for $J_\infty(\cdot)$ in $\mathcal 
	H(D,\R^2)$, that is,
	$
       	dJ_\infty(\Omega^\Gamma)(X)\ge 0$ for all  $X\in \mathcal H(\Omega,\R^2)$
if and only if
$
0 \in \bar L(\Omega^\Gamma).
$
\end{lemma}
\begin{proof}
	According to Lemma~\ref{lem:hilbert_projection}, we have 
	$0\not\in \bar L(\Omega^\Gamma)$ if and only if there exists $X_0\in \bar 
	L(\Omega^\Gamma)$, $X_0\ne 0$ satisfying
	$
       	(X_0,X_0)_{\mathcal H(\Omega,\R^d)} \ge 	(X_0,\varphi)_{\mathcal H(\Omega,\R^d)}$ for all $\varphi \in \bar L(\Omega^\Gamma). 
	$
	Thus $0\not\in \bar L(\Omega^\Gamma)$ implies
	$
       	-\|X_0\|^2_{\mathcal H(\Omega,\R^d)} \ge \max_{\varphi\in \bar L(\Omega^\Gamma)}(-X_0,\varphi)_{\mathcal H(\Omega,\R^d)} = dJ_\infty(\Omega^\Gamma)(-X_0)	
	$
        and hence $dJ_\infty(\Omega^\Gamma)(-X_0)<0$. Conversely if there exists 
	$X_0\in \bar L(\Omega^\Gamma)$, $X_0\ne 0$, such that $dJ_\infty(\Omega^\Gamma)(X_0)<0$, then
	$
       	(X_0,\varphi)_{\mathcal H(\Omega,\R^d)} \le dJ_\infty(\Omega^\Gamma)(X_0)<0$ for all  $\varphi\in \bar L(\Omega^\Gamma)
	$
	which can only be true if $0\not\in \bar L(\Omega^\Gamma)$. This finishes the proof. 	
\end{proof}

\begin{definition}
	We call $\gb\in\mathcal H(\Omega,\R^2)$ with $\|\gb\|_{\mathcal H}=1$ \emph{steepest descent direction} of $J_\infty(\cdot)$ at $\Omega$ if
	\ben
       	dJ_\infty(\Omega^\Gamma)(\gb) \le dJ_\infty(\Omega^\Gamma)(\varphi) \quad \text{ for all }\varphi\in \mathcal H(\Omega,\R^2) \text{ with } \|\varphi\|_{\mathcal H}=1.  
	\een
\end{definition}
\begin{remark}
	Since according \cite[Lemma 2.8]{HeineSturm16} the Eulerian semi-derivative $
	dJ_\infty(\Omega^\Gamma)(\cdot)$ is $1
	$-homogeneous we can restrict ourselves  to the unique sphere in the previous
	definition. 
\end{remark}

At this juncture let us introduce for $\Omega^\Gamma\in \Xi$ the function
\ben
\psi(\Omega^\Gamma) := 	\min_{\substack{X\in \mathcal H(\Omega,\R^2)\\ \|X\|_{\mathcal H}=1}}\max_{Z\in H(\Omega^\Gamma)}(Z,X)_{\mathcal H}. 
\een

\begin{lemma}\label{lem:minimum_steepest}
	Suppose that $\psi(\Omega^\Gamma)<0$. Then 
	\ben\label{eq:psi_le_0}
       	\psi(\Omega^\Gamma) = \max_{Z\in \bar L(\Omega^\Gamma)}\left(Z,-\frac{\hat Z}{\|\hat Z\|_{\mathcal H}}\right)_{\mathcal H} = - \|\hat Z\|_{\mathcal H},	
	\een
where $\hat Z\in \mathcal H$ solves the minimisation problem
$
\|\hat Z\|_{\mathcal H} = \min_{X\in \bar L(\Omega^\Gamma)}\|X\|_{\mathcal H}.	
$
\end{lemma}
\begin{proof}
	We apply Lemma~\ref{lem:hilbert_projection} with $H:=\mathcal H(\Omega,\R^2),$ $x_0=0$ and 
	$K=\bar L(\Omega^\Gamma)$. Hence we find $Z^*\ne 0$ in $\bar L(\Omega^\Gamma)$ with
	$ (Z,Z^*)_{\mathcal H} \ge (Z^*,Z^*)_{\mathcal H}$ for all $Z\in \bar L(\Omega^\Gamma).$
	Therefore for $\bar G:= Z^*/\|Z^*\|_{\mathcal H}$ it holds
	$ (\bar G, Z)_{\mathcal H} \le -\|Z^*\|_{\mathcal H}$ for all  $ Z\in \bar L(\Omega^\Gamma)	$
        which in turn implies 
	\ben\label{eq:le_L_Omega}
	\max_{Z\in \bar L(\Omega^\Gamma)} (Z^*,Z)_{\mathcal H} = 
	\left(Z^*,-\frac{Z^*}{\|Z^*\|_{\mathcal H}}\right)_{\mathcal H} = - \|Z^*\|_{\mathcal H}.
	\een
	This is already the second equality in \eqref{eq:psi_le_0}. 
	As for the second one we observe that Cauchy-Schwarz's inequalty shows 
	$(Z,Y)_{\mathcal H}\ge -\|Z\|_{\mathcal H}\|Y\|_{\mathcal H}$ for all 
	$Z,Y\in \mathcal H(\Omega,\R^2)$. 
	Hence we get
	$
       	-\|Z^*\|_{\mathcal H} \le (Z^*,G)_{\mathcal H} \le \max_{Z\in \bar L(\Omega^\Gamma)}(Z^*,G)_{\mathcal H}	
	$
for arbitrary $G\in \mathcal H(\Omega,\R^2)$ with $\|G\|_{\mathcal H}=1$. 
	So combining the previous inequality with  
	\eqref{eq:le_L_Omega} yields 
	the desired result. 
\end{proof}

\begin{lemma}\label{lem:func_l}
	The function 
	$
	l(G) := 	\max_{Z\in \bar 
		L(\Omega^\Gamma)}\left(Z,G\right)_{\mathcal H} 
	$ attains its minimum on the unit sphere in exactly one point.   
\end{lemma}
\begin{proof}
    The proof of \cite[Hilfsatz 3.3.7. p.63-44]{MR1088479} applies also to our setting.
\end{proof}

In analogy to the case in which Eulerian semi-derivative is linear (see \cite[Lemma~2.10]{eigelsturm16}) we can prove the existence and uniqueness
of steepest descent directions in the nonlinear case. However, the computation is more involed than in the linear case.
\begin{theorem}\label{eq:steepest_unique}
	Let $\Omega^\Gamma\in \Xi$ and suppose $\psi(\Omega^\Gamma)<0$. Then there is a unique steepest descent 
	direction $\gb\in \mathcal H(\Omega,\R^2)$, $\|\gb\|_{\mathcal H}=1$, for $J_\infty(\cdot)$ at $\Omega^\Gamma$ given by 
$
\gb := -\frac{\hat Z}{\|\hat Z\|_{\mathcal H}}, 
$
where $\hat Z = P_{\bar L(\Omega^\Gamma)}(0)$ is the projection of $0$ ( in 
$\mathcal H(\Omega,\R^2)$) onto $\bar L(\Omega^\Gamma)$.
\end{theorem}
\begin{proof}
	Follows from Lemmas~\ref{lem:func_l} and \ref{lem:minimum_steepest}. 
\end{proof}
\subsection{$\eps$-stationary points}
In order to define stable numerical algorithms we introduce $\eps$-stationary 
points. 
For $\eps \ge 0$ we define
 \begin{align}
	 R_\eps(\Omega^\Gamma) &  := \{y\in \overbar{\Omega}:\; J_\infty(\Omega^\Gamma)-\Psi(y,u(\Omega,\Gamma,y)) \le \eps\}\\	
	H_\eps(\Omega^\Gamma)      &  := \{\nabla j(\Omega^{\Gamma,y}):\; y\in R_\eps(\Omega^\Gamma) \}
\end{align}
and the convex hull of $H_\eps(\Omega^\Gamma)$ is denoted
by $L_\eps(\Omega^\Gamma)$.  Let us introduce for $\eps\ge 0$,
$$
\psi_\eps(\Omega^\Gamma) := 	\min_{\substack{X\in \mathcal H(\Omega,\R^2)\\ 
		\|X\|_{\mathcal H}=1}}\max_{y\in R_\eps(\Omega^\Gamma)}(\nabla j(\Omega^{\Gamma,y}),X)_{\mathcal H}. $$
Analogously to steepest descent directions we introduce $\eps$-steepest 
descent directions. 
\begin{definition}
   \begin{itemize}
	   \item[(i)] We call $\Omega^\Gamma \in \Xi $ an $\eps$-stationary point for $J_\infty(\cdot)
		   $ with respect to perturbations in  $\mathcal H(\Omega,\R^2)$ if $\psi_\eps(\Omega^\Gamma)\ge 0$. 
	   \item[(ii)]  
	We call $\gb_\eps$ a $\eps$-\emph{steepest descent direction} for $J_\infty(\cdot)
	$ at $\Omega^\Gamma$ in $\mathcal H(\Omega,\R^2)$ if
	\ben
       	\max_{Z\in \bar L_\eps(\Omega^\Gamma)} (Z,\gb_\eps)_{\mathcal H} =   \min_{\substack{X\in \mathcal H(\Omega,\R^2)\\ \|X\|_{\mathcal H }=1}} \max_{Z\in \bar L_\eps(\Omega^\Gamma)} (Z,X)_{\mathcal H}.  
	\een
   \end{itemize}
\end{definition}

It is readily verified that $\Omega^\Gamma$ is an $\eps$-stationary point if 
and only if $0\in \bar L_\eps(\Omega^\Gamma)$. 
Moreoever, if 
$\psi_\eps(\Omega^\Gamma)<0$ there  is a  unique $\eps$-steepest
		descent direction at $\Omega^\Gamma$ in the space $\mathcal H(\Omega,\R^2)$ given by 
		$
		\gb=-\frac{\hat \gb_\eps}{\|\gb_\eps\|_{\mathcal H}},
		$
		where $\gb_\eps$ is the projection of $0$ onto 
		$\bar L_\eps(\Omega^\Gamma)$.

The crucial point of $\eps$-steepest descent directions $g_k$ is that they decrease 
the cost function $J_\infty(\cdot)$. Suppose that 
$
\psi_{\eps_k}(\Omega^\Gamma) < 0
$
and let $\gb_k :=\gb_k^\eps$ be the $\eps_k$-steepest descent direction. 
Then
\ben
dJ_\infty(\Omega^\Gamma)(\gb_k)\le \max_{Z\in \bar L_\eps(\Omega^\Gamma)} (Z,\gb_k)_{\mathcal H}<0
\een
and as a consequence
$J_\infty(\Phi_t^{\gb_k}(\Omega^\Gamma)) < J_\infty(\Omega^\Gamma)$ for sufficiently 
small $t$. 
The parameter $\eps$ is a sort of regularisation parameter and ensures that the 
steepest descent directions are not ``too local''.

\subsection{Discrete problems}\label{sec:discrete_problems}

\subsubsection*{Discretisation of the domain $\Omega$}
We assume that $\Omega$ is a polygonial set. 
Let  $\{\mathcal T_h\}_{h>0}$ denote a family of simplicial
triangulations $\mathcal T_h=\{K\}$ consisting 
of triangles $K$ such that 
$$
\overbar{\Omega} = \bigcup \limits_{K\in \mathcal T_h}  K, \quad \forall h >0.
$$
For every element $K\in \mathcal T_h$, $h(K)$ stands for the diameter of  $K$ 
and $\rho(K)$ for the diameter of the largest ball contained in $K$. The maximal 
diameter of all elements is denoted by $h$, i.e., 
$h:=\textnormal{max} \{h(K) \  | \ K\in \mathcal T_h\}.$ Each 
$K\in \mathcal T_h$ consists of three nodes and three edges and we denote the 
set of nodes and edges by 
$\mathcal N_h$ and $\mathcal E_h$, respectively. We assume that 
there exists a positive constant $\varrho >0$, independent of $h$,   such 
that 
$\frac{h(K)}{\rho(K)} \le \varrho$
holds for all elements $K \in \mathcal T_h$ and all $h>0$. 

\subsubsection*{Discrete $\eps$-steepest descent directions and the quadratic program}
In order to obtain an algorithm we select for fixed $\eps >0$ a finite subset 
$R^h_\eps(\Omega^\Gamma)\subset R_\eps(\Omega^\Gamma)$ of points. We use 
the triangulation of $\Omega$ as discretisation, that is,
\ben
R^h_\eps(\Omega^\Gamma) := R_\eps(\Omega^\Gamma) \cap \mathcal N_h = \{y_1,\ldots, y_{\Neh}\}. 
\een
We have $\#(R^h_\eps(\Omega))=\Neh. $
Let us set 
$H^h_\eps(\Omega) := \{\nabla j(\Omega^{\Gamma,y}):\; y\in R^h_\eps(\Omega)\}$ and denote by 
$L^h_\eps(\Omega)$ the convex hull of $H^h_\eps(\Omega)$.  
For $y\in R^h_\eps(\Omega)$ we introduce the vectors $X_k := \nabla j(\Omega^{\Gamma,y_k})$ 
and order them $\{X_1,\cdots, X_{\Neh}\}$. For simplicity set henceforth 
$N:=\Neh$ and keep in mind that $N$ depends on $\eps$ and $h$. In order to obtain steepest descent directions in $\mathcal H(\Omega,\R^2)$, we need 
to solve 
$
 \min_{X\in L^h_\eps(\Omega)}\|X\|_{\mathcal H}. 
$
Using the definition of $L^h_\eps(\Omega)$ we see that this task is equivalent to 
solving the quadratic problem 
\ben\label{eq:quadratic}
\min \sum_{k,l =1}^N \alpha_k \alpha_l (X_k, X_l)_{\mathcal H}\quad \text{ 
	subject to }\sum_{k=1}^N \alpha_{k} =1, \quad  \alpha_k \ge 0. 
\een
Defining $Q_N:= ((X_k,X_l)_{\mathcal H})_{l,k=1,\ldots,N}$, $G_N:=(1,\ldots, 1)$, 
$E_N := -I$, $g_N=-(1,\ldots, 1)^\top$ and
$\alpha = (\alpha_1,\ldots, \alpha_N)^\top$, problem \eqref{eq:quadratic} can
be written in the canonical form:
$
	\min_{\alpha }\;  Q_N\alpha \cdot \alpha
$ subject to $B_N\alpha =0 $ and  $E_N\alpha \le g_N. $
This quadratic problem is convex and thus admits a unique solution 
$\alpha^*=(\alpha^*_1,\ldots, \alpha_N^*)$. The $\eps$-steepest descent direction is given by
\ben\label{eq:eps_descent}
\gb_\eps^h := - \frac{\hat\gb^*_\eps}{\|\hat\gb^*_\eps\|_{\mathcal H}}, \quad \text{where 
}\quad \hat \gb^*_\eps := \sum_{k=1}^N \alpha^*_k \nabla j(\Omega^{\Gamma,y_k}).	
\een

\begin{remark}
It is clear that we are not obliged to use the triangulation of $\Omega$ to 
construct a discretisation for $R_\eps^h(\Omega^\Gamma)$, however,  it is 
advantageous from 
the practical point of view. 
\end{remark}
\section{Numerical realisation}\label{sec:numerics}
\subsection{Problem setting}
We consider two cost functions
\ben\label{eq:L2_Linfty}
J_\infty(\Omega) := \max_{x\in \overbar \Omega} |u(x)-u_d(x)|^2, \qquad J_2(\Omega) := \int_\Omega |u-u_d|^2\; dx, 
\een
where in either case $u$ is the solution of ($x=(x_1,x_2)$)
\ben\label{eq:test_state}
\begin{split}
	- \Delta u(x) + u(x) & = (2 \pi^2 +1)\sin(\pi x_1) \sin(\pi x_2) \quad   \text{ in } \Omega,\\
	u(x) & = 0 \quad  \text{ on } \partial \Omega. 
\end{split}
\een
Notice that we set $\Omega:= \Omega^\emptyset$ since  $\Gamma=\emptyset$. We now 
define $u_d(x):= \sin(\pi x_1) \sin(\pi x_2) $, such that by construction
$
\Omega_{opt}\in \text{argmin}J_2 
$
and 
$\Omega_{opt}\in \text{argmin}J_\infty
$
with $\Omega_{opt} := (0,1)\times (0,1)$. 
Indeed the unique solution of \eqref{eq:test_state} on 
$(0,1)\times (0,1)$ reads 
$u(x)=\sin(\pi x_1) \sin(\pi x_2)$ as can be readily verified. By the properties of the sinus function we see that also every other domain 
$\Omega_n := (2n, 2n+1)\times (2n, 2n+1)$, $n\in \Z$ is a global minimum of $J_\infty$ and 
$J_2$, respectively.

\subsubsection*{Finite element approximation}
Now we describe discretisations of $dJ_2(\Omega)$ and 
$dJ_\infty(\Omega)$.  
Let $V_h(\Omega)$, $h>0$, denote the usual $H^1(\Omega)$ conforming finite element 
space, that is,
\ben\label{eq:V_h_k}
 V_h( \Omega) := \{v\in C(\overbar{\Omega}):\; v_{|K}\in \mathcal P^1(K), 
	\;\forall  K\in \mathcal T_h \}.	
\een
By $\ac V_h(\Omega)$ we denote all function of the space $V_h(\Omega)$ that vanish on 
the boundary $\partial \Omega$. 
For 
each $y\in \overbar \Omega$ the finite element approximation 
$(u_{h},p_h^y)\in \ac V_h(\Omega)\times \ac V_h(\Omega)$ of state 
\eqref{eq:state} and adjoint state equation 
\eqref{eq:adjoint} reads,
\begin{align}
	\int_\Omega \nabla u_h \cdot \nabla \varphi + u_h \varphi\, dx  & =  \int_\Omega f  
	\varphi  \, dx  \quad \text{ for all } \varphi \in \ac V_h(\Omega) \label{eq:approx_state_adjoint1} \\ 
	\int_\Omega \nabla \varphi \cdot \nabla p_{y,h} +  \varphi p_{y,h} \, dx  & = - 2(u(y)-u_h(y))\varphi(y) \quad \text{ for all } \varphi \in \ac V_h(\Omega).\label{eq:approx_state_adjoint2}
\end{align}
With the discretised state and adjoint state equation the discretised version 
of \eqref{eq:SD_adjoint}, (where $\Psi(y,\zeta):= |\zeta - u_d(y)|^2$) reads
\ben\label{eq:vol_discrete}
dJ_\infty^{h}(\Omega)(X) = \max_{y\in R^h(\Omega)}\left(\int_{\Omega} 
	\Sb_1^{y,h}:\partial X + \Sb_0^{y,h}\cdot X \;dx  - X(y)\cdot \nabla u_d(y) 2(u_h(y)-u_d(y)) \right),	
\een
where for $y\in R^h_\eps(\Omega)$ we set $\Sb_1^{y,h} := \Sb_1(u_h,p_{y,h})$   
and $\Sb_0^{y,h} := \Sb_0(u_h,p_{y,h})$ with $\Sb_1,\Sb_2$ being defined in 
\eqref{eq:formula_S1},\eqref{eq:formula_S0} (with $\beta\equiv 1$).

The shape derivative of 
$
J_2(\Omega) = \int_\Omega|u-u_d|^2\; dx,	
$
subject to $u$ solves \eqref{eq:test_state}, in an open and bounded subset 
$\Omega\subset D$ in direction   $X\in \ac C^1(D,\R^2)$ (see 
\cite{sturm2015shape} for the computation) is given by
\ben\label{eq:dJ_2_cont}
dJ_2(\Omega)(X) = \int_\Omega \Tb_1(u,\hat p):\partial X + \Tb_0(u,\hat p)\cdot X \; dx,	
\een
where $\hat p$ solves the adjoint equation
\ben
\int_\Omega \nabla \hat p \cdot \nabla \varphi + \hat p \varphi \; dx = -\int_\Omega 2(u-u_d) \varphi \; dx 
\quad \text{ for all } \varphi\in \ac H^1(\Omega).
\een
The tensors $\Tb_1$ and $\Tb_2$ are given by 
$\Tb_1(u,\hat p)     :=    (|u-u_d|^2 -f \hat p  +  \nabla u 
	\cdot \nabla \hat p)I - (\nabla u \otimes \nabla \hat p + \nabla \hat p
	\otimes \nabla u)$
	and 
	$
\Tb_0(u,\hat p)  := - \nabla f \hat p - 2\nabla u_d (u-u_d).	
$
The discrete version of \eqref{eq:dJ_2_cont} reads 
\ben\label{eq:dJ_2_dis}
dJ_2^h(\Omega)(X) = \int_\Omega \Tb_1(u_h,\hat p_h):\partial X + \Tb_0(u_h,\hat p_h)\cdot X \; dx,	
\een
where the discrete state $u_h$ solves \eqref{eq:approx_state_adjoint1} and the 
discrete adjoint state $\hat p_h\in \ac V_h(\Omega)$ solves:
\ben
\int_\Omega \nabla \hat p_h \cdot \nabla \varphi + \hat p_h \varphi \; dx = -\int_\Omega 2(u_h-u_d) \varphi \; dx 
\quad \text{ for all } \varphi\in \ac V_h(\Omega).
\een

\subsubsection{Choice of the metric}
We run our numerical tests with two different metrics on the space 
$V_h(\Omega)\times V_h(\Omega)$, namely the $H^1$ metric and the 
Euclidean metric.  Let $v^1,\ldots, v^{2N}$ be a basis of $V_h(\Omega)\times V_h(\Omega)$ and 
$\alpha_i,\beta_i$ in $\R$, $i,j=1,2,\ldots, 2N$ and suppose
$
v=\sum_{i=1}^{2N} \alpha_i v^i$ 
and
$
 w= \sum_{i=1}^{2N} \beta_i v^i. 
$
The $H^1$ metrc and Eulcidean metric are defined by
$$
(v, w )_{H^1}:=  \sum_{i,j=1}^{2N}\alpha_i\beta_j M_{ij}, \quad \left(v, w \right)_{\mathcal V_h} := \sum_{i,j=1}^{2N}\alpha_i\beta_j \delta_{ij}, 
$$
where $M_{ij}$ is defined by 
$
 M_{ij} = \int_\Omega \partial v^i : \partial v^j + v^i \cdot v^j\; dx
$
and $\delta_{ij}$ denotes the 
Kronecker delta. 
 We denote by $\mathcal H_{\text{Sob}} $ and $\mathcal H_{\text{Euc}}$  
the space $V_h(\Omega)\times V_h(\Omega)$ equipped with the $H^1$ and 
Euclidean metric, respectively. 
Both spaces are kernel reproducing Hilbert spaces and thus the point evaluation is continuous; see 
\cite[Section 3]{eigelsturm16}.  

The approximated Eulerian semi-derivative \eqref{eq:vol_discrete}
can equivalently be written as:
\ben\label{eq:eulerian_FE}
\begin{split}
dJ_\infty^{h}(\Omega)(X) = \max_{y\in R^h(\Omega)} (\nabla^{\text{Euc}} 
j^h(\Omega^{y}),X)_{\mathcal H_{\text{Euc}}} = \max_{y\in R^h(\Omega)} (\nabla^{\text{Sob}} j^h(\Omega^{y}),X)_{\mathcal H_{\text{Sob}}},
\end{split}
\een
where for all $y\in R^h(\Omega)$ the gradient 
$\nabla^{\text{Sob}} j^h(\Omega^{y})$ is defined as the solution of
\ben\label{eq:gradient_j_omega}
(\nabla^{\text{Sob}} j^h(\Omega^{y}), \varphi)_{\mathcal H_{\text{Sob}}}  = \int_{\Omega} \Sb_1^{y,h}:\partial \varphi + \Sb_0^{y,h}\cdot \varphi \;dx  -     2\nabla u_d(y)\cdot \varphi(y) (u_h(y)-u_d(y)) 	
\een
for all 
$\varphi \in  V_h(\Omega) \times  V_h(\Omega) $. The gradient 
$\nabla^{\text{Euc}} j^h(\Omega^{y})$ is explicitly given by
\ben\label{eq:gradient_j_omega2}
\nabla^{\text{Euc}} j^h(\Omega^{y}) =  \sum_{k=1}^{2N} \left(\int_{\Omega} \Sb_1^{y,h}:\partial v^k + \Sb_0^{y,h}\cdot v^k \;dx  -     2\nabla u_d(y)\cdot v^k(y) (u_h(y)-u_d(y))\right)v^k. 
\een
We refer to \cite[Section 3]{eigelsturm16} for more details. The 
advantage of the Euclidean metric is that it does no require the solution of a 
variational problem but only the evaluation of the shape derivative 
$dj^h(\Omega^y)(v^j)$ at the basis elements $v^j$.

Similarly, the discretised shape derivative of $J_2$ can be written as
\ben
dJ_2^h(\Omega)(X) =  (\nabla^{\text{Euc}} 
J_2^h(\Omega),X)_{\mathcal H_{\text{Euc}}} = (\nabla^{\text{Sob}} J_2^h(\Omega^{y}),X)_{\mathcal H_{\text{Sob}}},
\een
where
\ben\label{eq:gradient_j_omega3}
(\nabla^{\text{Sob}} J_2^h(\Omega), \varphi)_{\mathcal H_{\text{Sob}}}  = 
\int_{\Omega} \Tb_1^{h}:\partial \varphi + \Tb_0^{h}\cdot \varphi \;dx  \quad 
\text{ for all } \varphi \in  V_h(\Omega) \times  V_h(\Omega) 	
\een
and 
\ben\label{eq:gradient_j_omega4}
\nabla^{\text{Euc}} J_2^h(\Omega) =  \sum_{k=1}^{2N} \left(\int_{\Omega} 
	\Tb_1^{h}:\partial v^k + \Tb_0^{h}\cdot v^k \;dx\right)v^k. 
\een

We will use the notation $J^h_2(\Omega):= \int_\Omega |u_h-u_d|^2\; dx$ and 
$J^h_\infty(\Omega) = \max_{x\in \overbar \Omega}|u_h(x)-u_d(x)|^2$. 

\subsection{Steepest descent algorithm}
The following algorithm uses the discretisation described in Section~\ref{sec:discrete_problems}. 

\begin{algorithm}[H]
    
    
    
     \KwData{Let $n=0$, $h>0$,$\gamma>0$ and $N\in \N$ be given. Initialise domain 
	     $\Omega_0\subset \R^2$. Let $N_2>0 $.
}     
     \While{ $n  \le  N $}{
                      1.) choose $\eps$, so that, $\#(R^h(\Omega_n))  \le \#(R^h_\eps(\Omega_n)) \le N_2 + \#(R^h(\Omega_n))$ \;
             2.) solve \eqref{eq:approx_state_adjoint1} to get $u_h$ and \eqref{eq:approx_state_adjoint2} for all $y\in R^h_\eps(\Omega_n)$ to obtain $p_{y,h}$\; 
	     3.) solve \eqref{eq:gradient_j_omega} to obtain  gradients
	     $\{\nabla j(\Omega^{y_1}),\ldots, \nabla j(\Omega^{y_{\Neh}})\}$\;
	     4.) solve quadratic program \eqref{eq:quadratic} to obtain $\gb_\eps^h$ 
	     defined in \eqref{eq:eps_descent}\;
	     5.) decrease $t$ until
	     \ben
          J_\infty^h( (\text{id}+t\gb_\eps^h)(\Omega_n)) < J^h_\infty(\Omega_n) 
          \een
           and set $\Omega_{n+1} := (\text{id}+t\gb_\eps^h)(\Omega_n) $  \;
          \eIf{ 
		  $J^h_\infty(\Omega_n) - J^h_\infty(\Omega_{n+1}) \ge \gamma (J^h_\infty(\Omega_0) - 
		  J^h_\infty(\Omega_1))$\;  }{increase   $n\rightarrow n+1$ and continue program\;  
               }{abort algorithm, no sufficient decrease}     
           }
    
     \caption{$\eps$-steepest descent algorithm} \label{algo:steepest}
    \end{algorithm}

\subsection{Numerical simulations}
The state equation, adjoint state equation and the shape derivative are discretised as described 
in \eqref{eq:approx_state_adjoint1},\eqref{eq:approx_state_adjoint2} and \eqref{eq:vol_discrete}, respectively. The domain 
$\Omega$ consists in each iteration of around 5500 nodes and we remesh in each 
iteration step. The boundary $\partial \Omega$ itself is discretised with a 
fixed number of 400 nodes which are moved during the optimisation process.  
As initial domain we choose a circle centered at $x=(0.5,0.5)$ with radius $r=\sqrt 6\approx 2.44$. 

In Figure~\ref{fig:steepest_Linfty} and Figure~\ref{fig:steepest_Linfty2} several iterations of Algorithm~\ref{algo:steepest}
applied to $J_\infty^h(\cdot)$ are displayed.  
The blue points indicate points of the triangulation
contained in $R^h_\eps(\Omega_n)$, where $n$ is the current iteration number. 
The number $N_2$ in Algorithm~\ref{algo:steepest} is chosen to be between $40$ 
and $80$. We did not perform a linesearch, that means, step four in the 
algorithm is replace by choosing a constant step size.  It can be seen that the optimal shape 
is quite good approximated using: (i) the $H^1$ metric in 
Figure~\ref{fig:steepest_Linfty} and (ii) the Euclidean metric in 
Figure~\ref{fig:steepest_Linfty2}. Even the corners are reconstructed 
quite well. Observe that
the points in $R^h_\eps(\Omega_n)$ are mostly distributed on the boundary 
$\partial \Omega$, so that for those points no adjoint equation has to be 
computed (cf. Corollary~\ref{cor:boundary}). 
In Figure~\ref{fig:steepest_L2} and Figure~\ref{fig:steepest_L22} we applied \cite[Algorithm 1]{eigelsturm16} to the cost function $J_2^h$. We use the same discretisation as before.

In Figure~\ref{eq:cost_log_plot_L2_Linfty_euclidean} the values $J_2^h(\Omega_n)$ 
over the number of iteration are plotted both in log scale.  For the dashed lines we used the 
$H^1$ metric and for the solid lines the Euclidean metric. 
 It makes sense to 
replace  $J_\infty(\Omega_n)$ by 
$J_2^h(\Omega_n)$ as a measure for the convergence rate
as the latter cost function can be estimated 
by $J_\infty(\Omega_n)$ using H\"older's inequality. 
We observe 
that the convergence rate in the 
 smooth case (minimising $J_2^h(\cdot)$) $J_2^h(\Omega_n)$ is slower than in the nonsmooth case (minimising $J_\infty^h(\cdot)$)  $J_2^h(\Omega_n)$.
 In the nonsmooth case the convergence rate even speeds up again in later iterations. That in the 
nonsmooth case corners do not perfectly match might have the reason that
the $\eps$ in our algorithm does not tend to zero in the end as we keep 
 $N_2\ge 40$. In fact it is difficult to find a reasonable condition to 
 decrease $\eps$. In the numerical practice it seems better to keep 
the number $N_2$ fixed in order to obtain a stable algorithm. 
The $H^1$ metric yields
smoother shapes than the Euclidean metric in general.

%

\begin{figure}
    \centering
\includegraphics[width=0.24\textwidth]{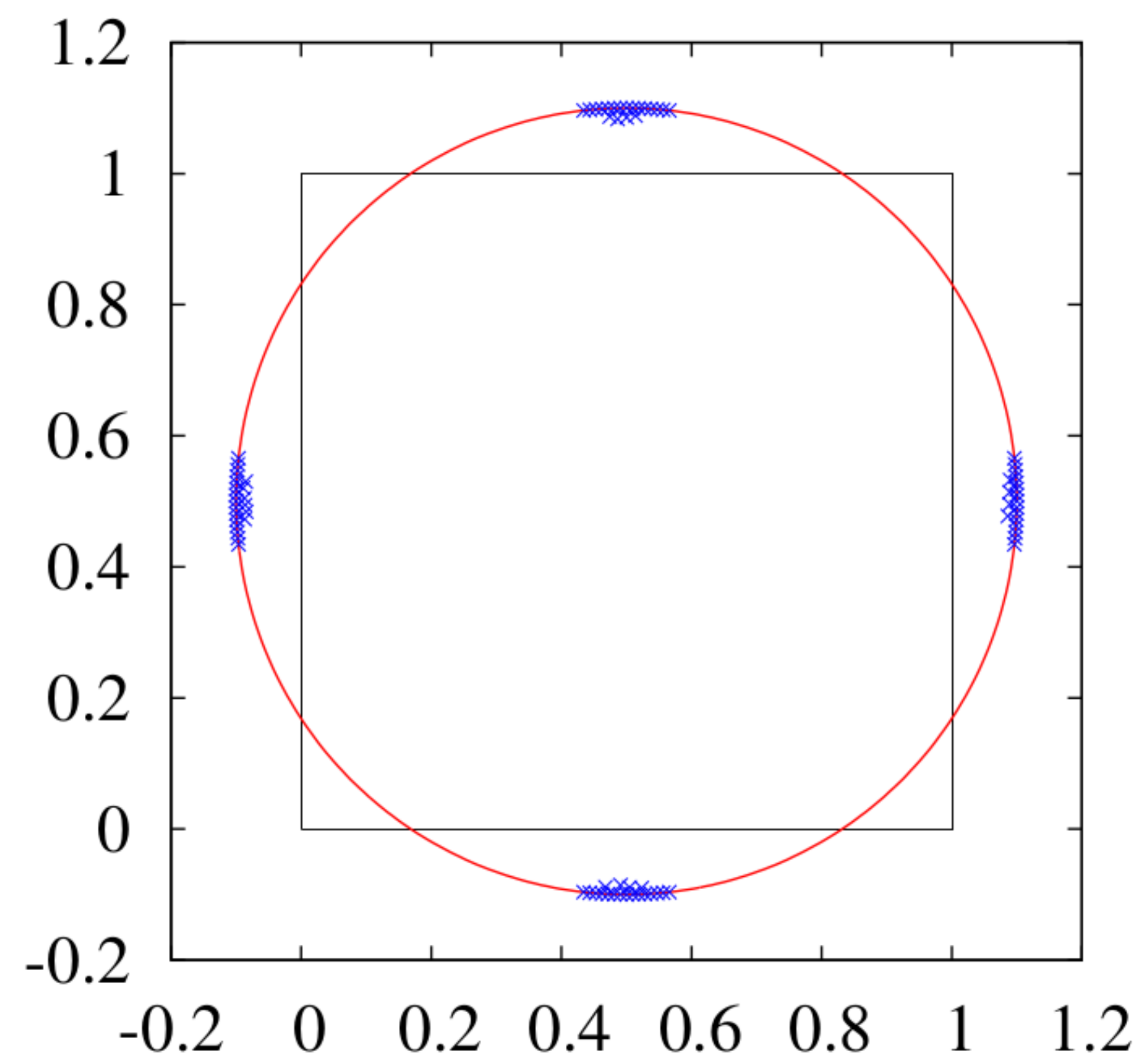}
\includegraphics[width=0.24\textwidth]{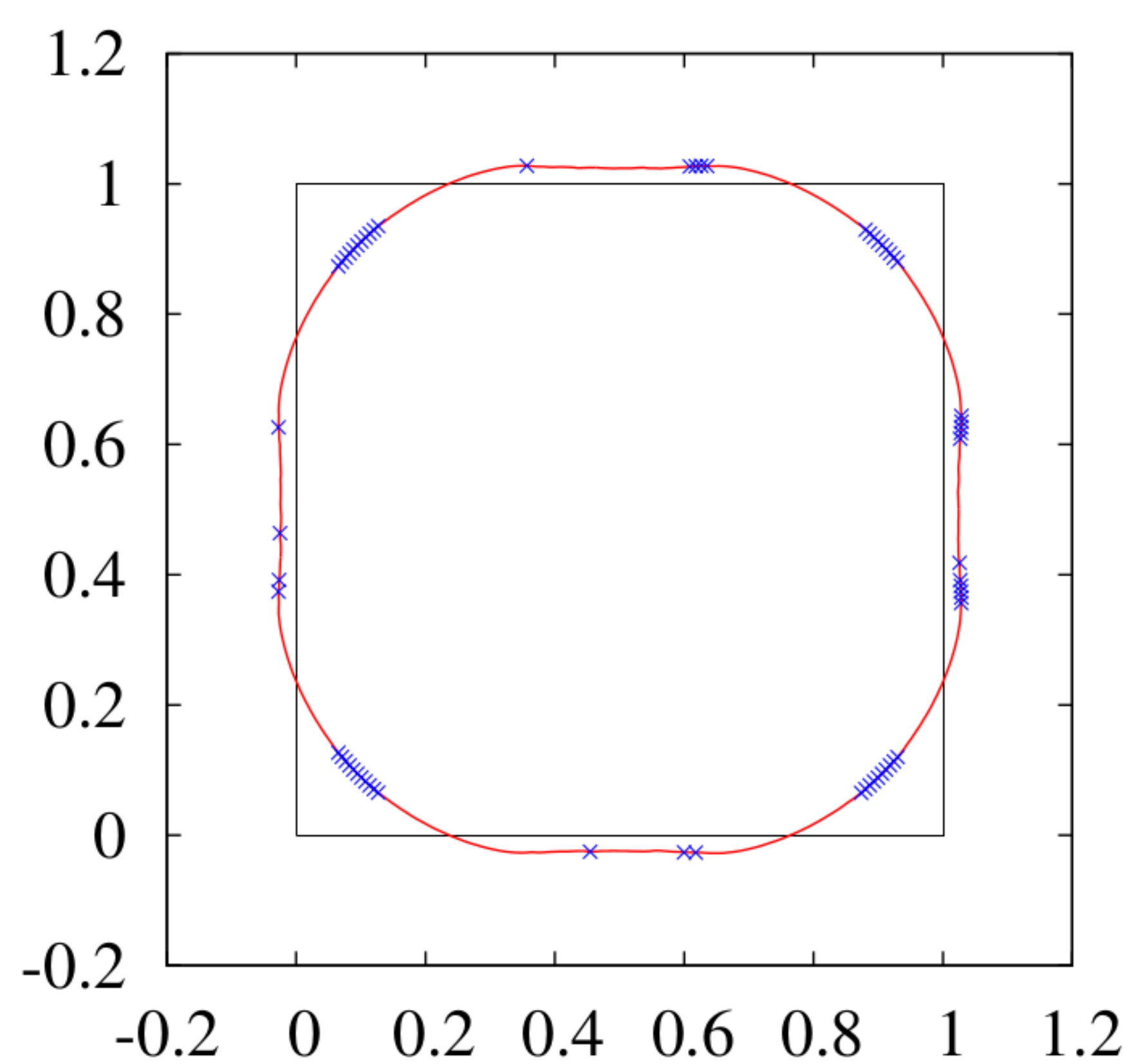}
\includegraphics[width=0.24\textwidth]{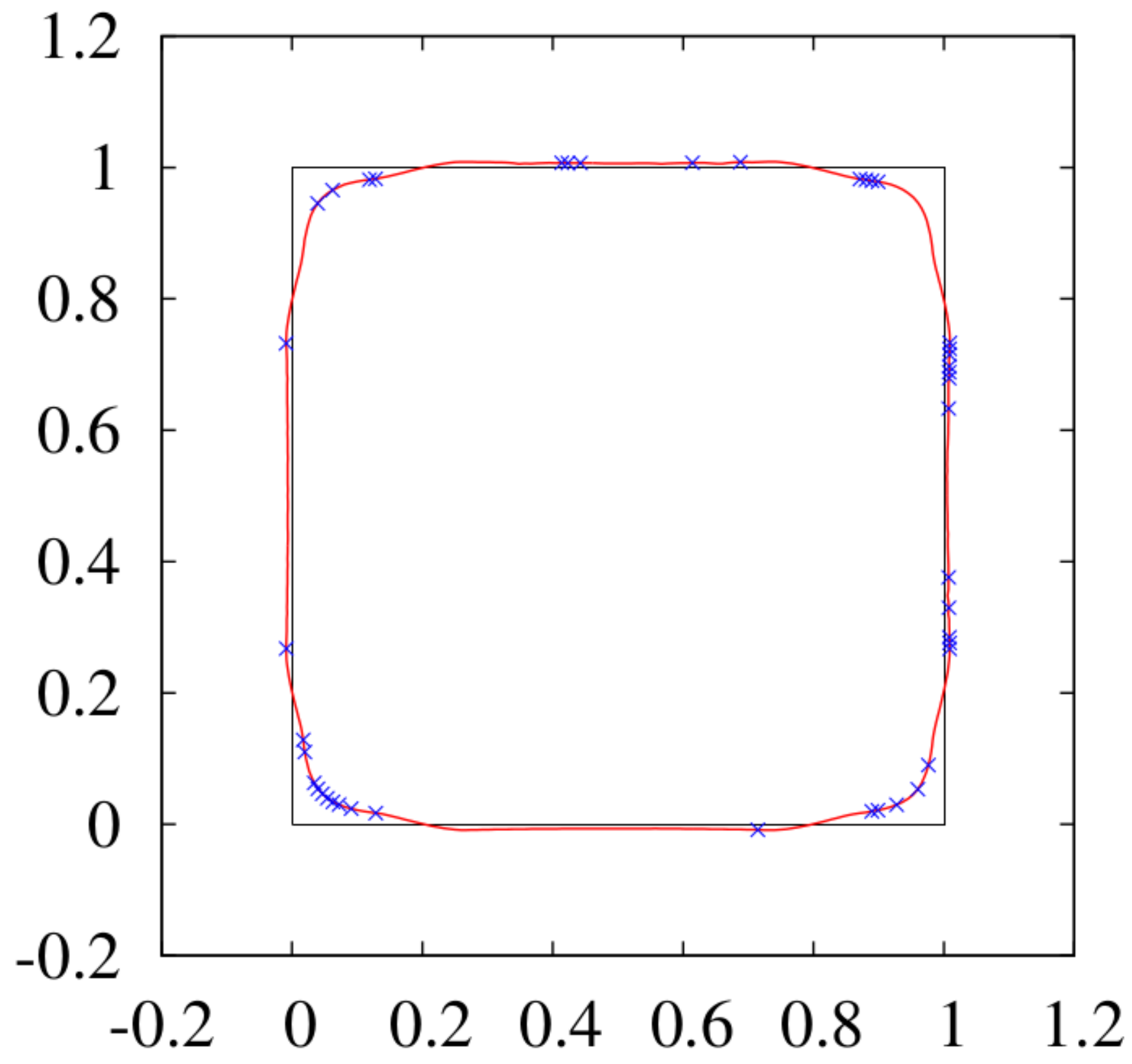}
\includegraphics[width=0.24\textwidth]{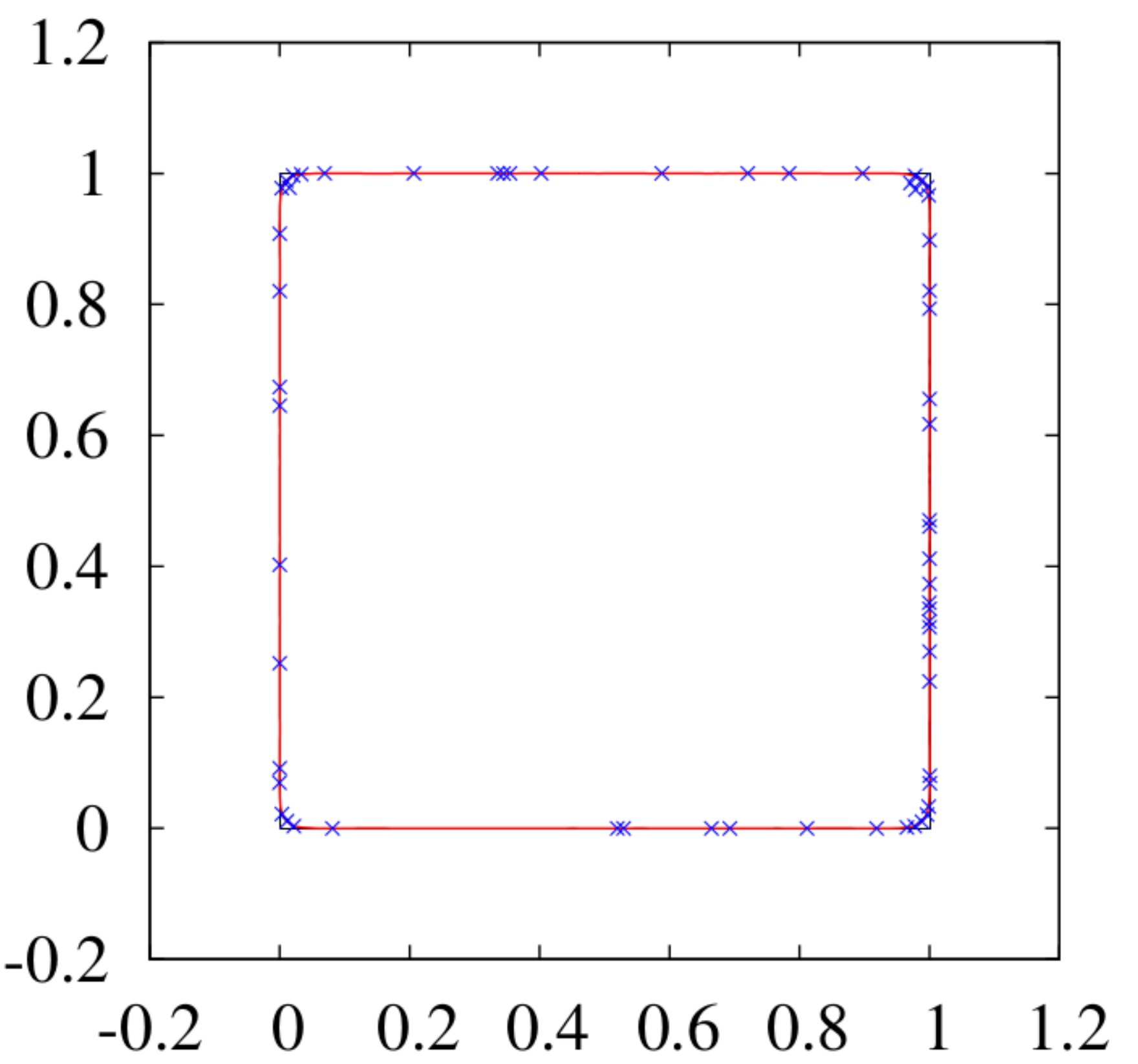}
\caption{Results for $J_\infty^h$ with $H^1$ metric; blue: 
	     points in $R_\eps^h(\Omega_n)$; red: boundary of shape $\partial \Omega_n$; from left to right: 
	initial shape, iteration 10, 120, 2000}
\label{fig:steepest_Linfty}
\end{figure}


\begin{figure}
    \centering
\includegraphics[width=0.24\textwidth]{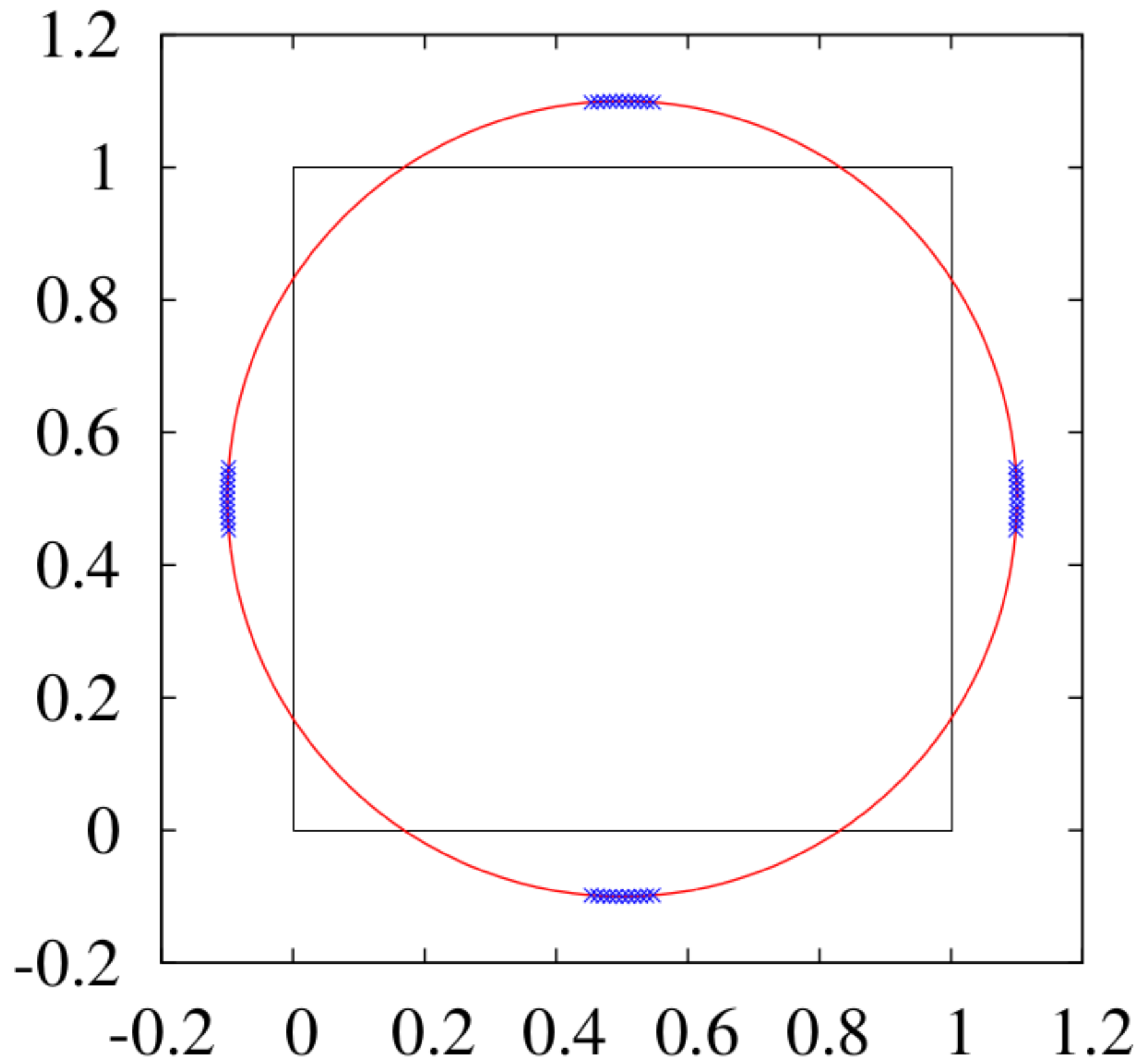}
\includegraphics[width=0.24\textwidth]{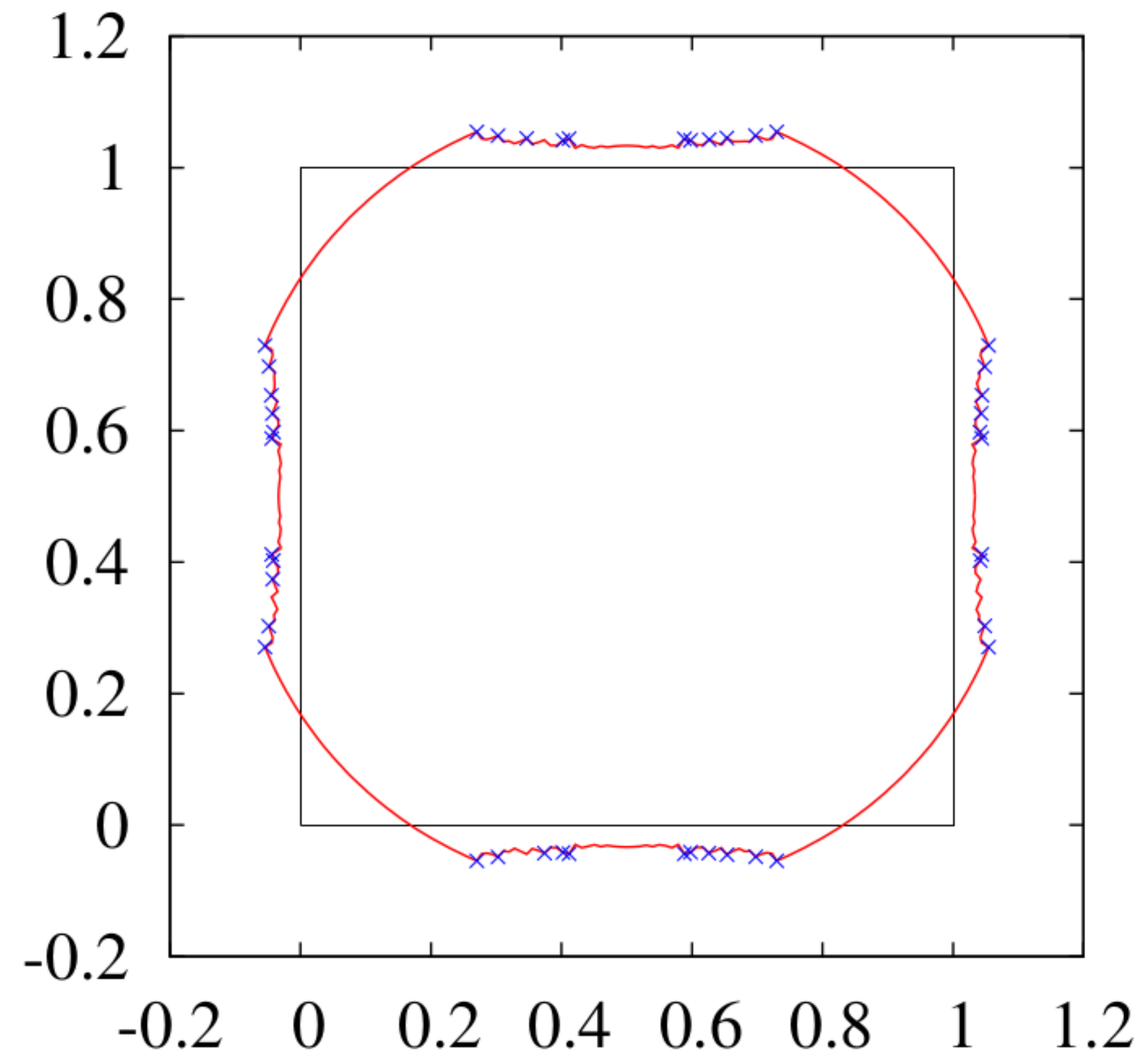}
\includegraphics[width=0.24\textwidth]{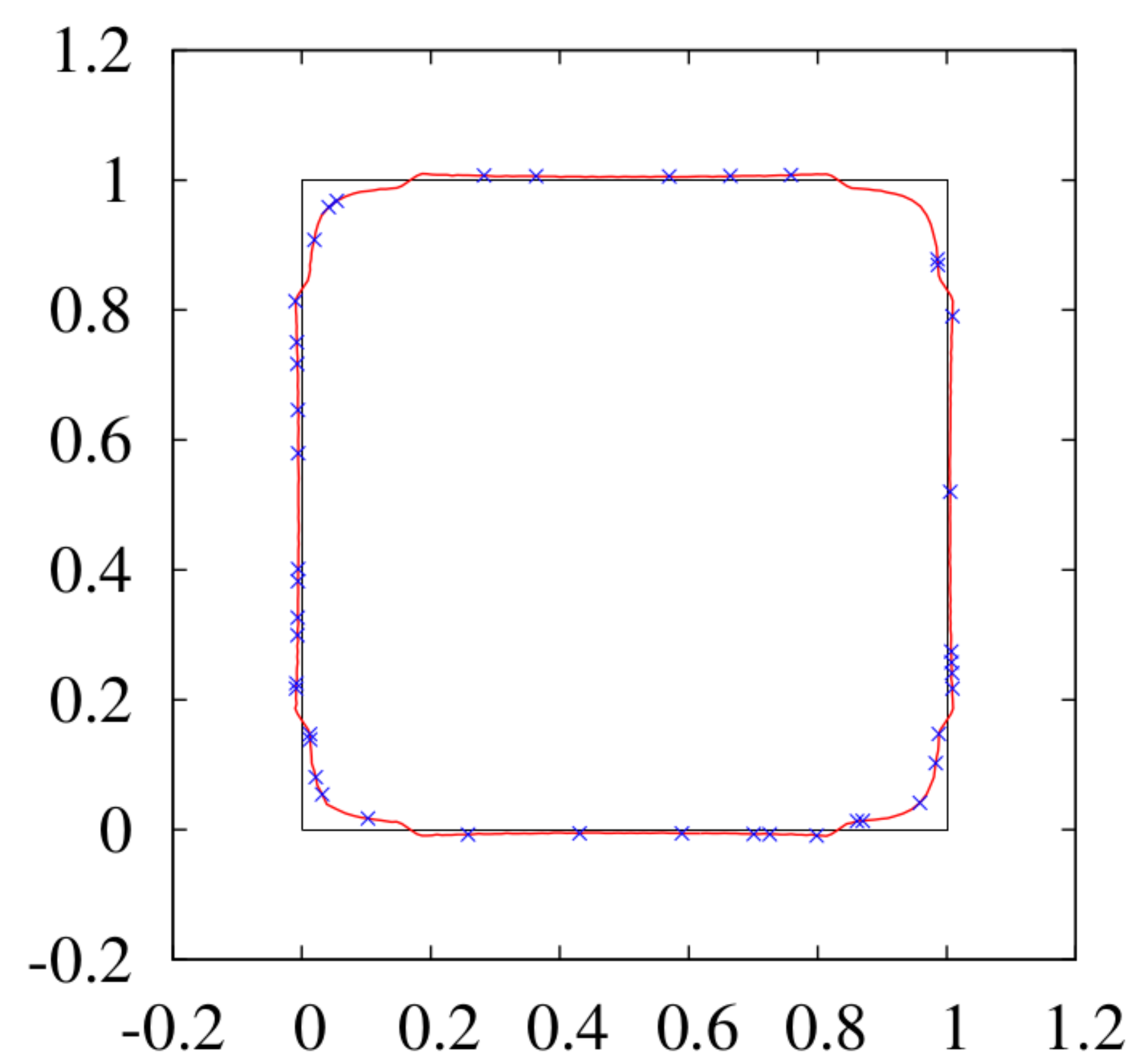}
\includegraphics[width=0.24\textwidth]{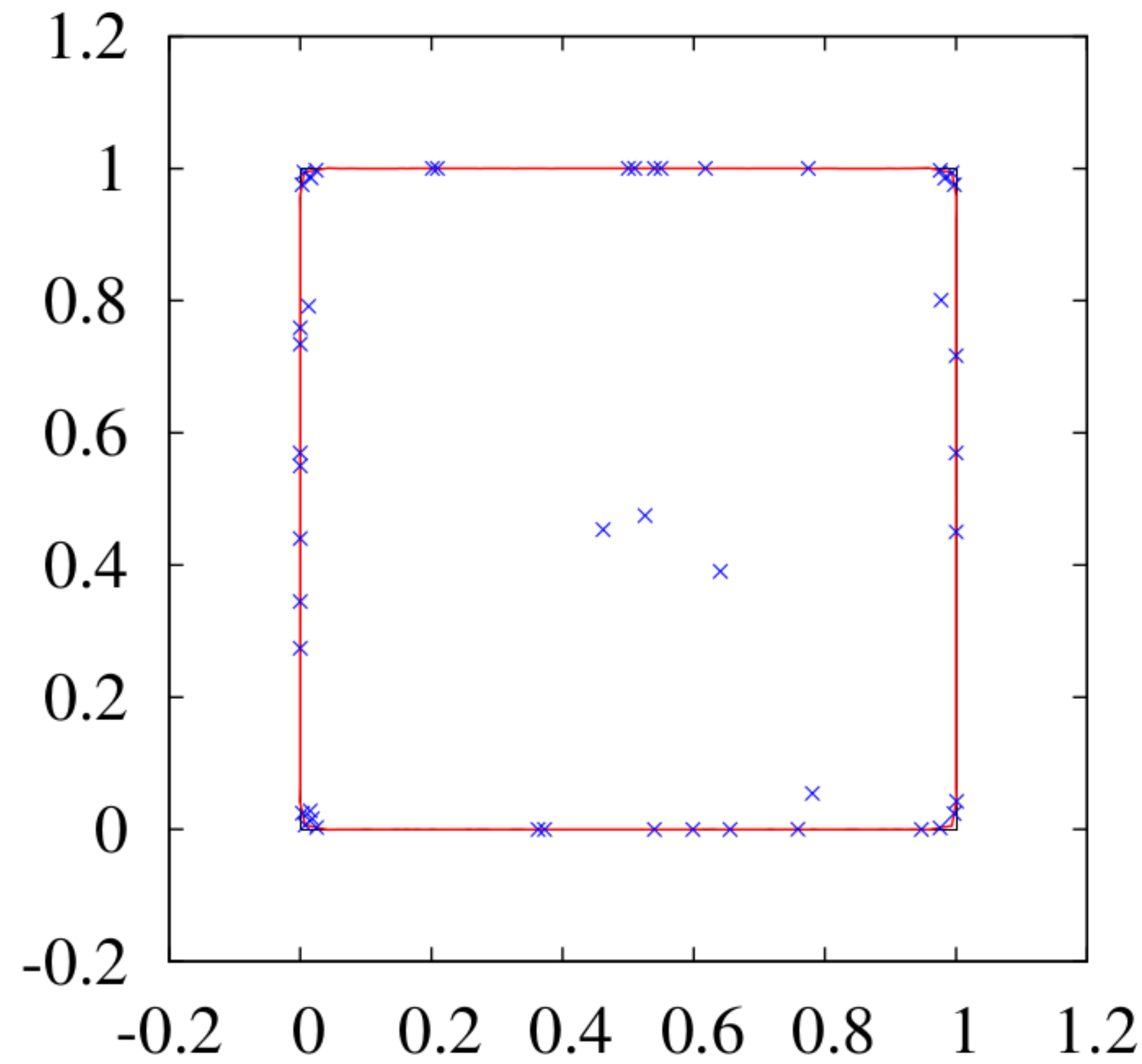}
\caption{Results for $J_\infty^h$ with Euclidean metric; blue: 
	     points in $R_\eps^h(\Omega_n)$; red: boundary of shape $\partial \Omega_n$; from left to right: 
	initial shape, iteration 10, 100, 2000}
\label{fig:steepest_Linfty2}
\end{figure}


 \begin{figure}
    \centering
\includegraphics[width=0.24\textwidth]{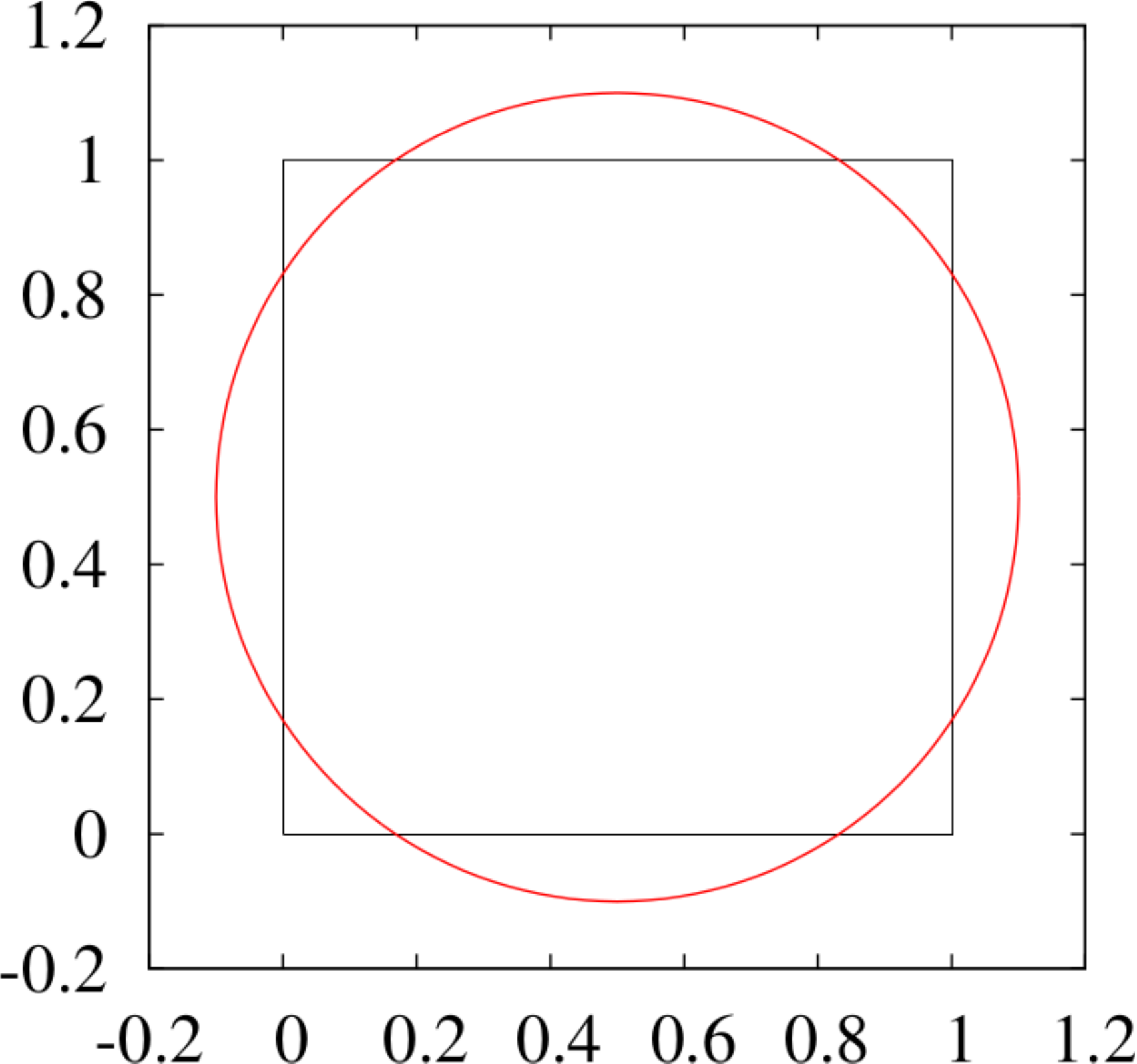}
\includegraphics[width=0.24\textwidth]{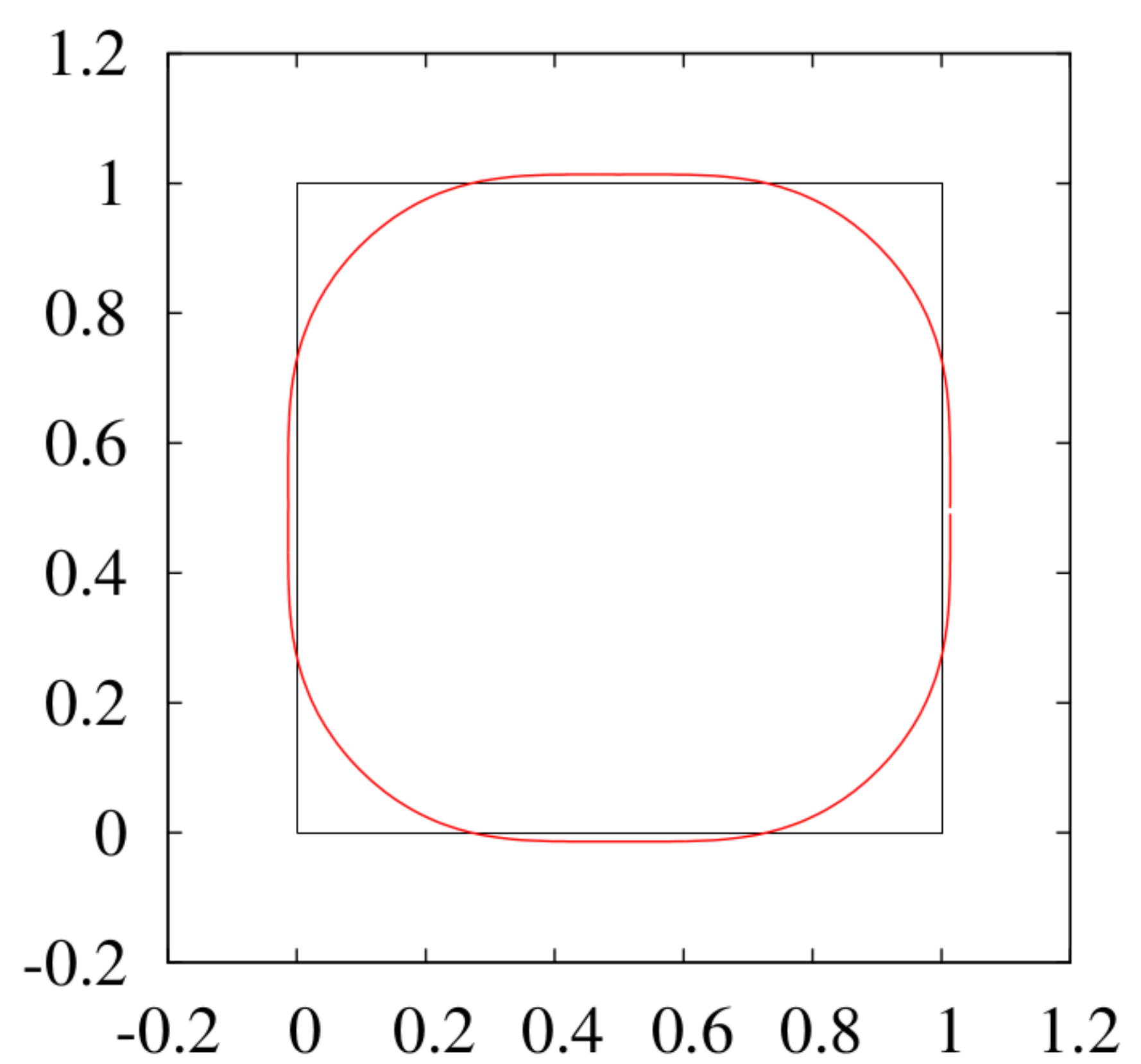}
\includegraphics[width=0.24\textwidth]{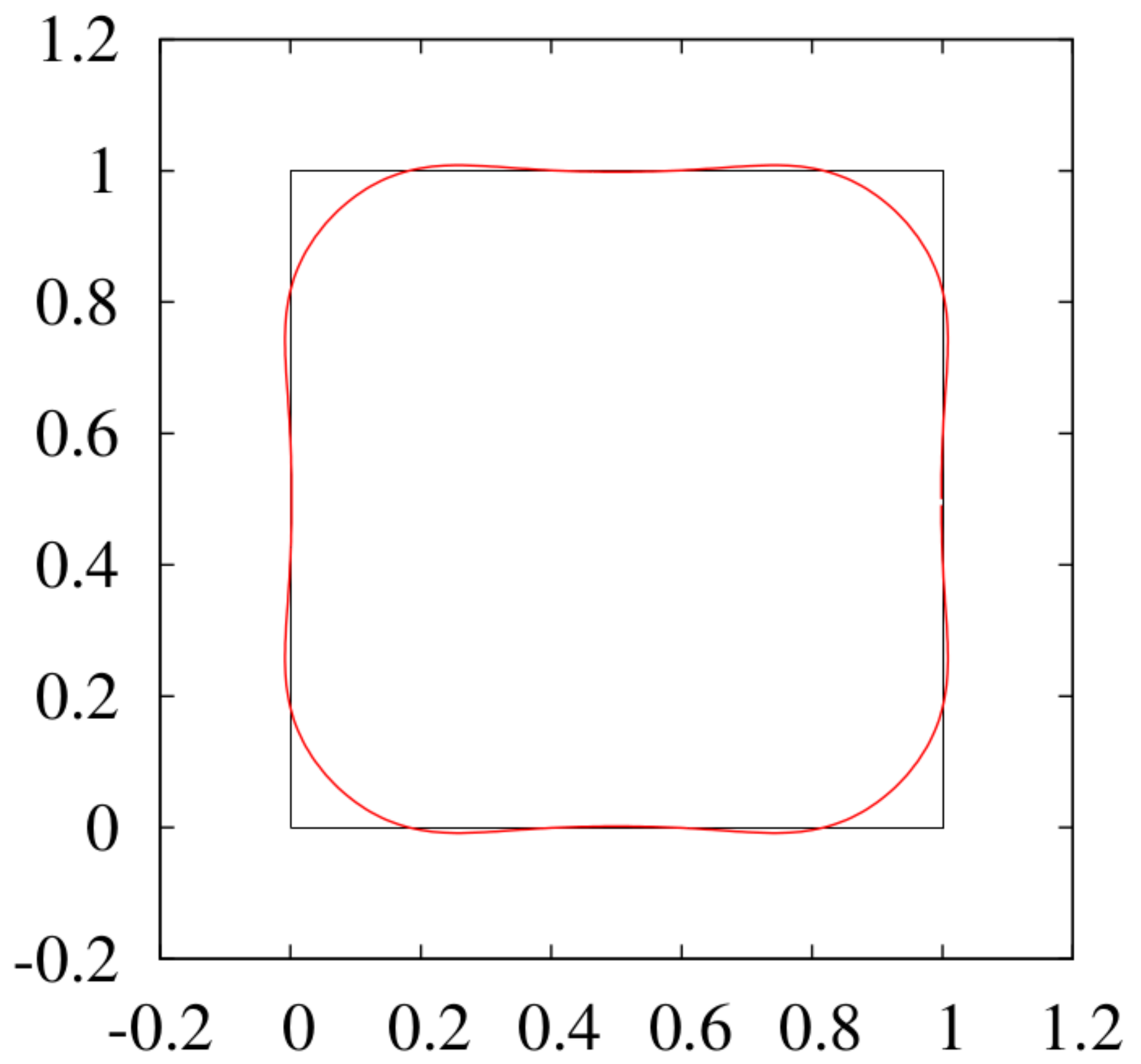}
\includegraphics[width=0.24\textwidth]{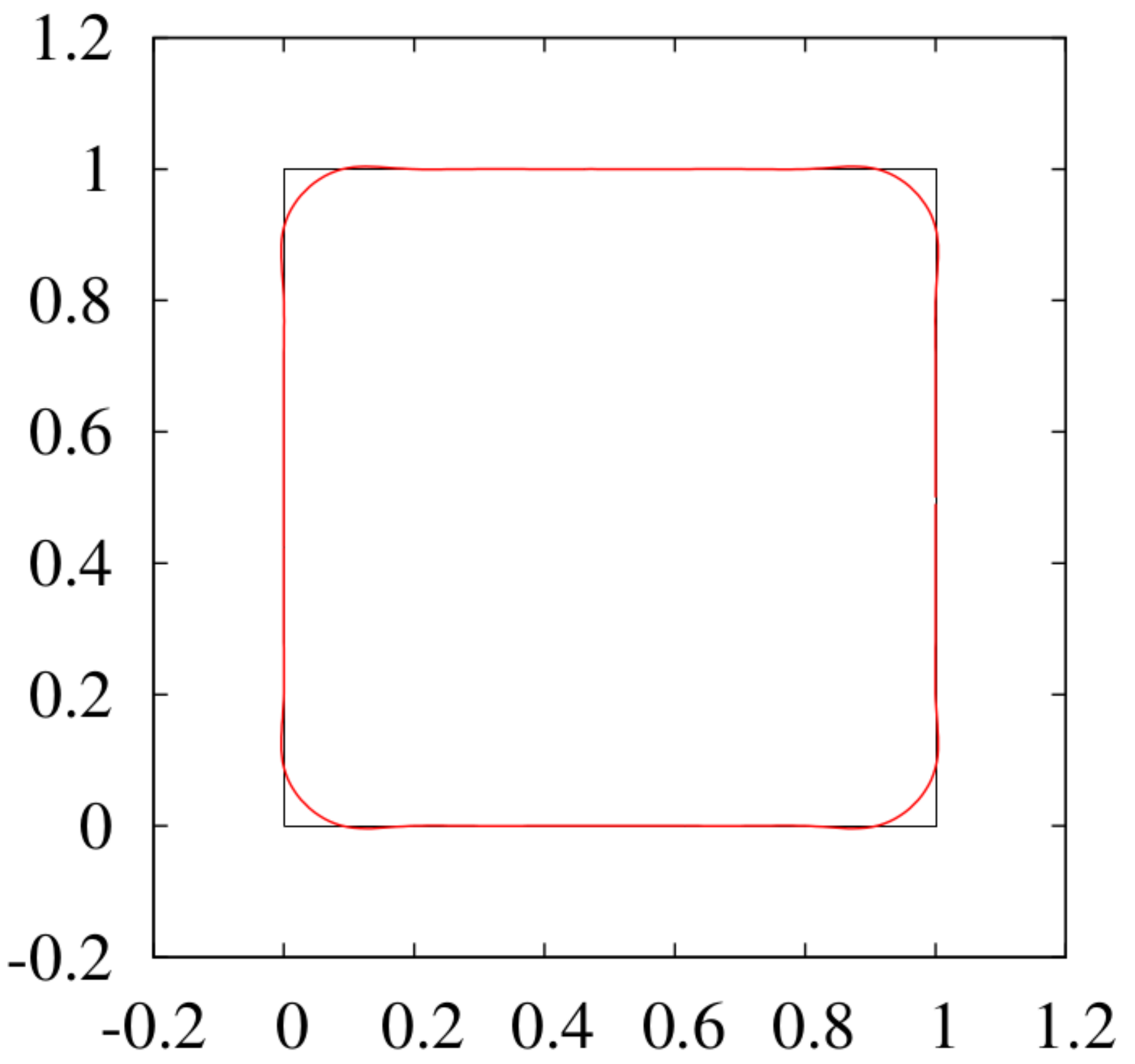}
\caption{Results for $J_2^h$ with $H^1$ metric; from left to right: 
	initial shape, iteration 20, 120, 2000}
\label{fig:steepest_L2}
\end{figure}


 \begin{figure}
    \centering
\includegraphics[width=0.24\textwidth]{./plots/L2_init.pdf}
\includegraphics[width=0.24\textwidth]{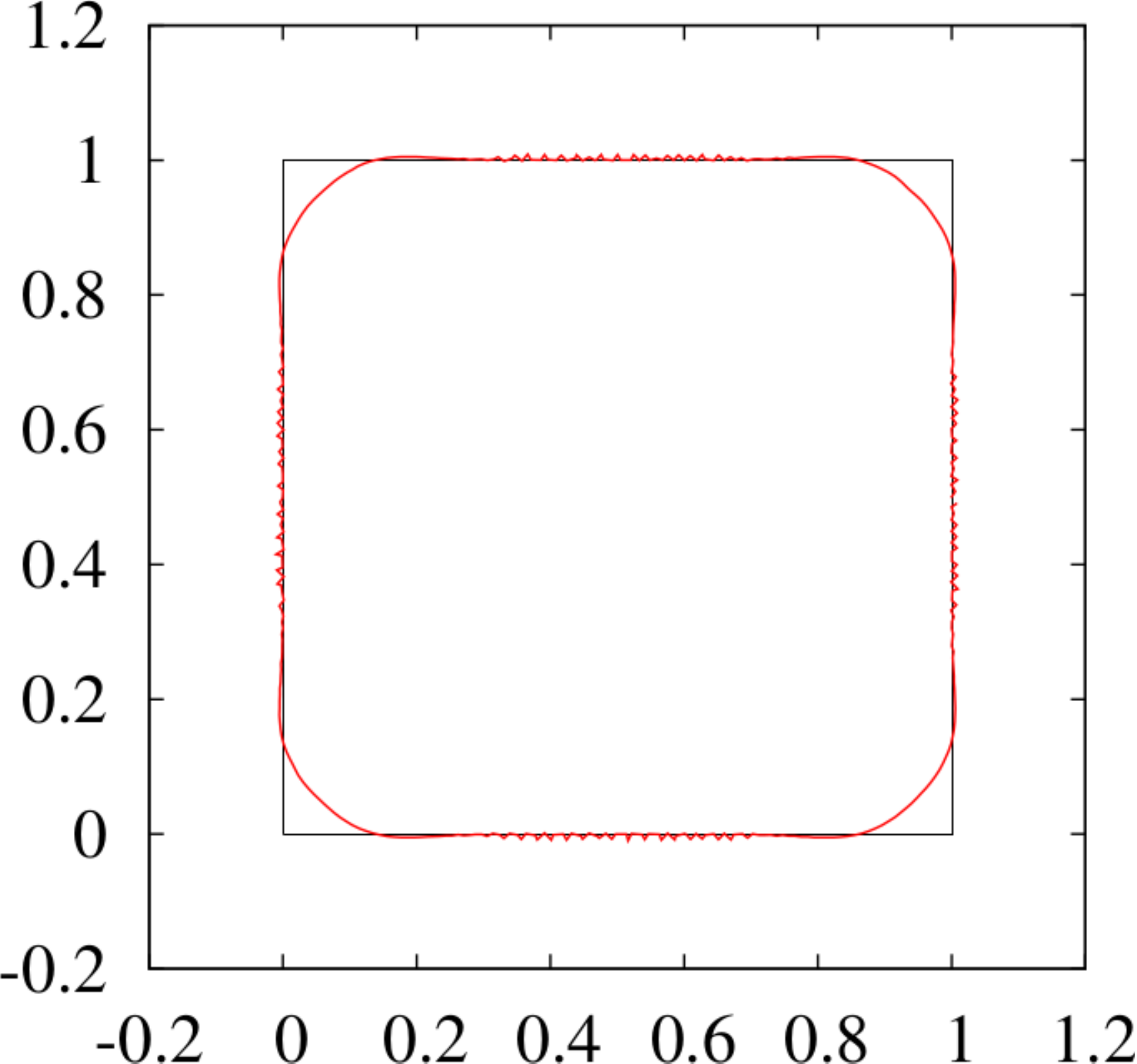}
\includegraphics[width=0.24\textwidth]{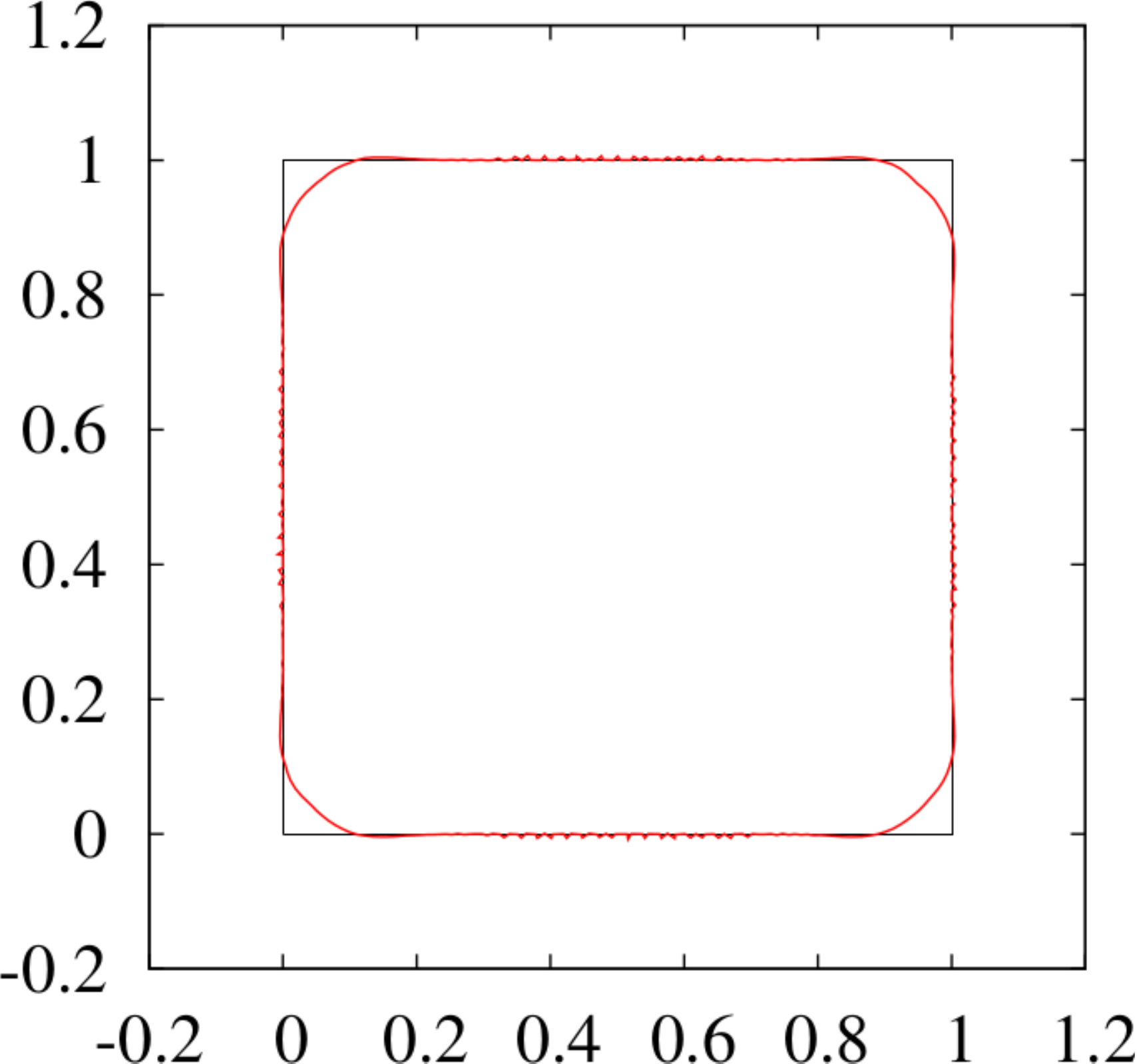}
\includegraphics[width=0.24\textwidth]{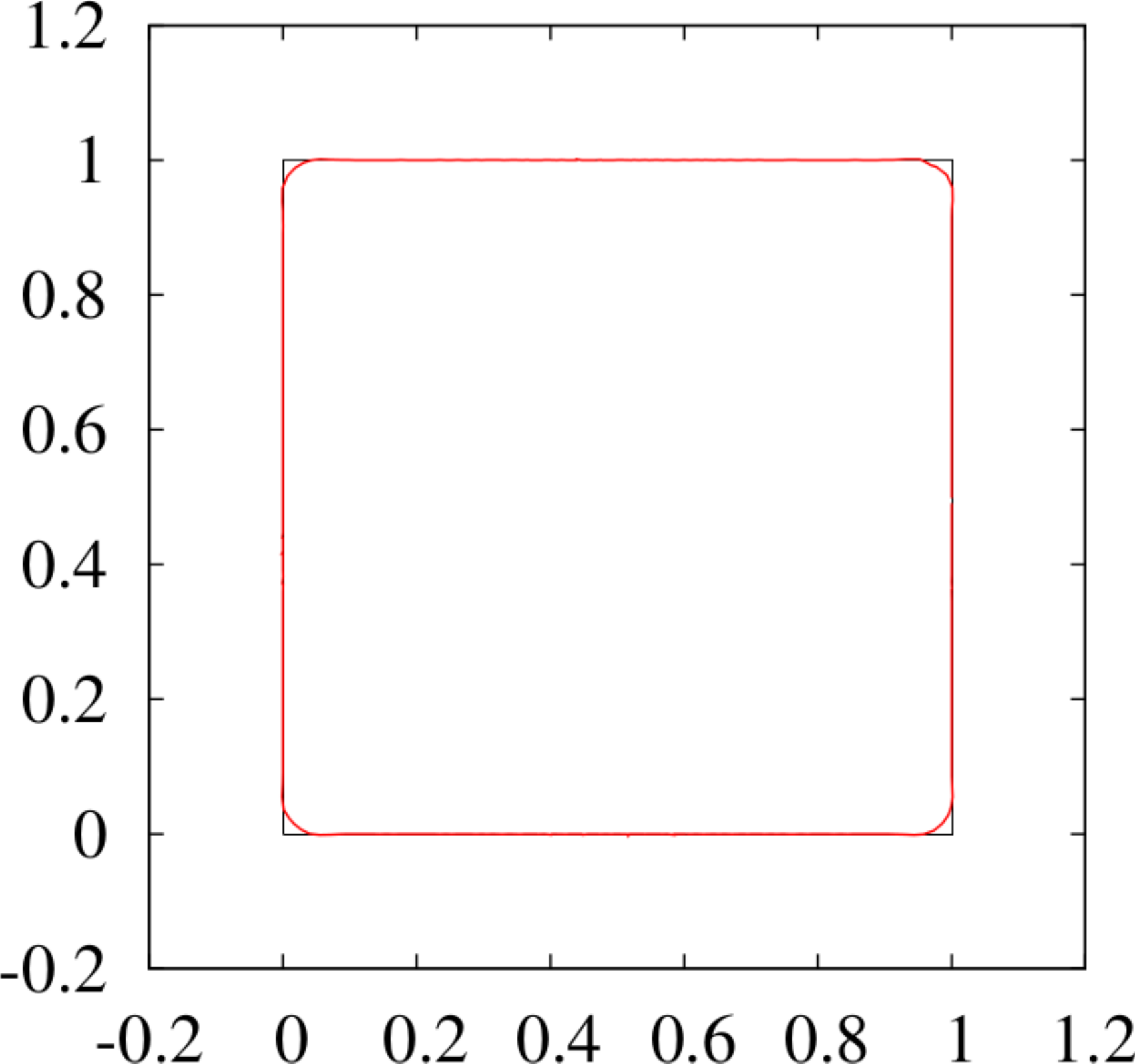}
\caption{Results for $J_2^h$ with Euclidean metric; from left to right: 
	initial shape, iteration 60, 120, 2000}
\label{fig:steepest_L22}
\end{figure}

Notice that to compute one descent direction for $J^h_\infty$ we have to solve at least one state equation,
$\#(R^h_\eps(\Omega_n))-\#(\Gamma_0)$ adjoint state equations and $\#(R^h_\eps(\Omega_n))$ gradient equations $\nabla j(\Omega^{y}_n)$. However the computation is \emph{perfectly 
	parallel}, that means, the computation of the adjoints and gradients 
can be parallized. Another possibility to reduce the computational cost is to 
use the boundary expression \eqref{eq:boundary_dJ_infty}, but the 
accuracy of this expression is lower than the domain expression \eqref{eq:SD_adjoint}.
In fact after discretisation \eqref{eq:boundary_dJ_infty} and  \eqref{eq:SD_adjoint}
are not equivalent anymore; cf. \cite{eigelsturm16}.
In contrast, to compute a steepest descent direction for $J_2^h$ only 
one state equation, one adjoint equation and one shape gradient has to be 
computed.


\begin{figure}[H]
    \centering
    \begin{subfigure}{.38\textwidth}
        \centering
	\scalebox{0.27}{ \includegraphics{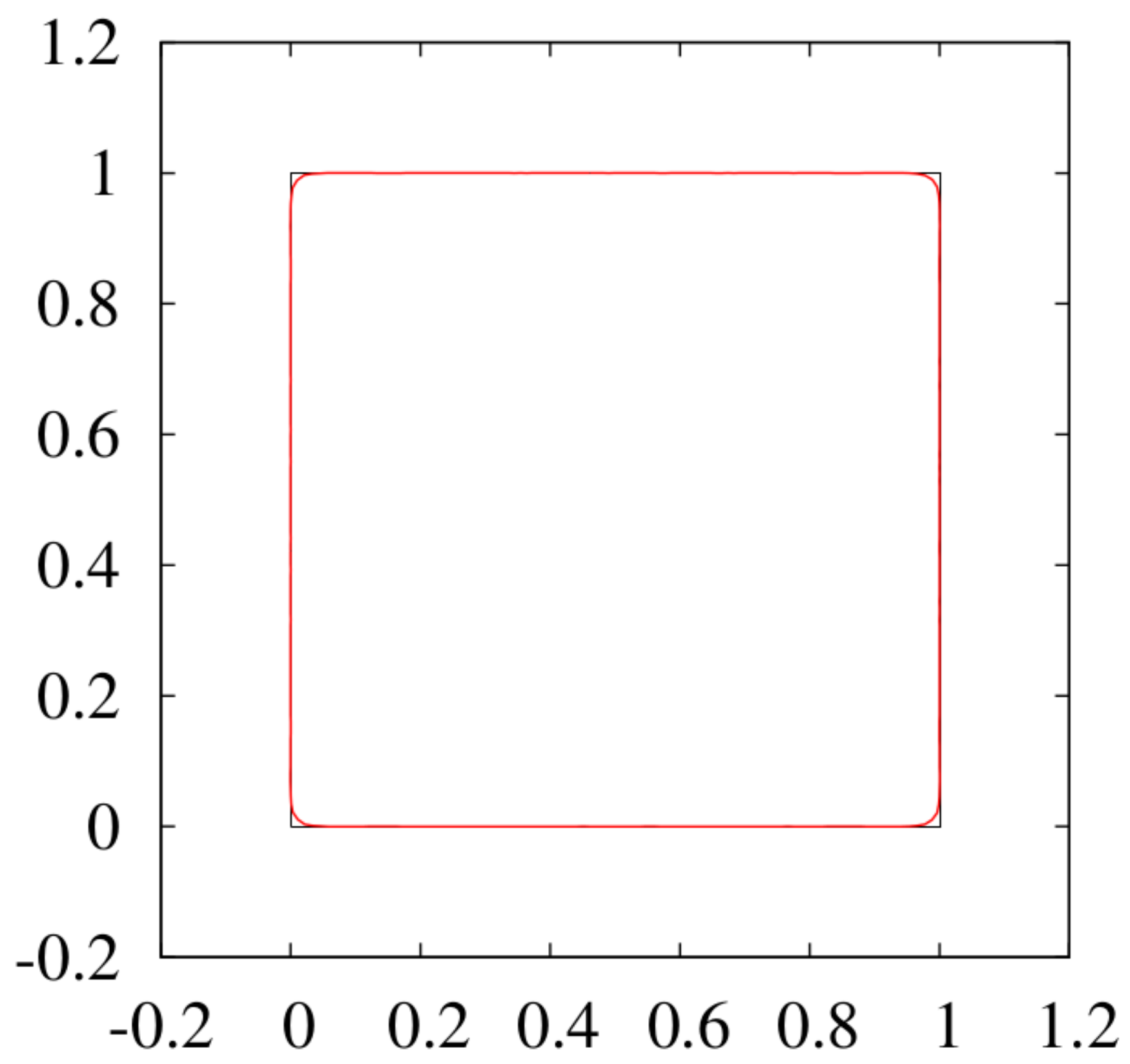}}
	\caption{$J_\infty^h$ with $H^1$ metric}
    \end{subfigure}%
    \begin{subfigure}{.38\textwidth}
        \centering
	\scalebox{0.27}{\includegraphics{./plots/iter_L2_2000_euclid.pdf}}
	\caption{$J_2^h$ with Euclidean metric}
     \end{subfigure}
    \begin{subfigure}{.38\textwidth}
        \centering
	\scalebox{0.27}{ \includegraphics{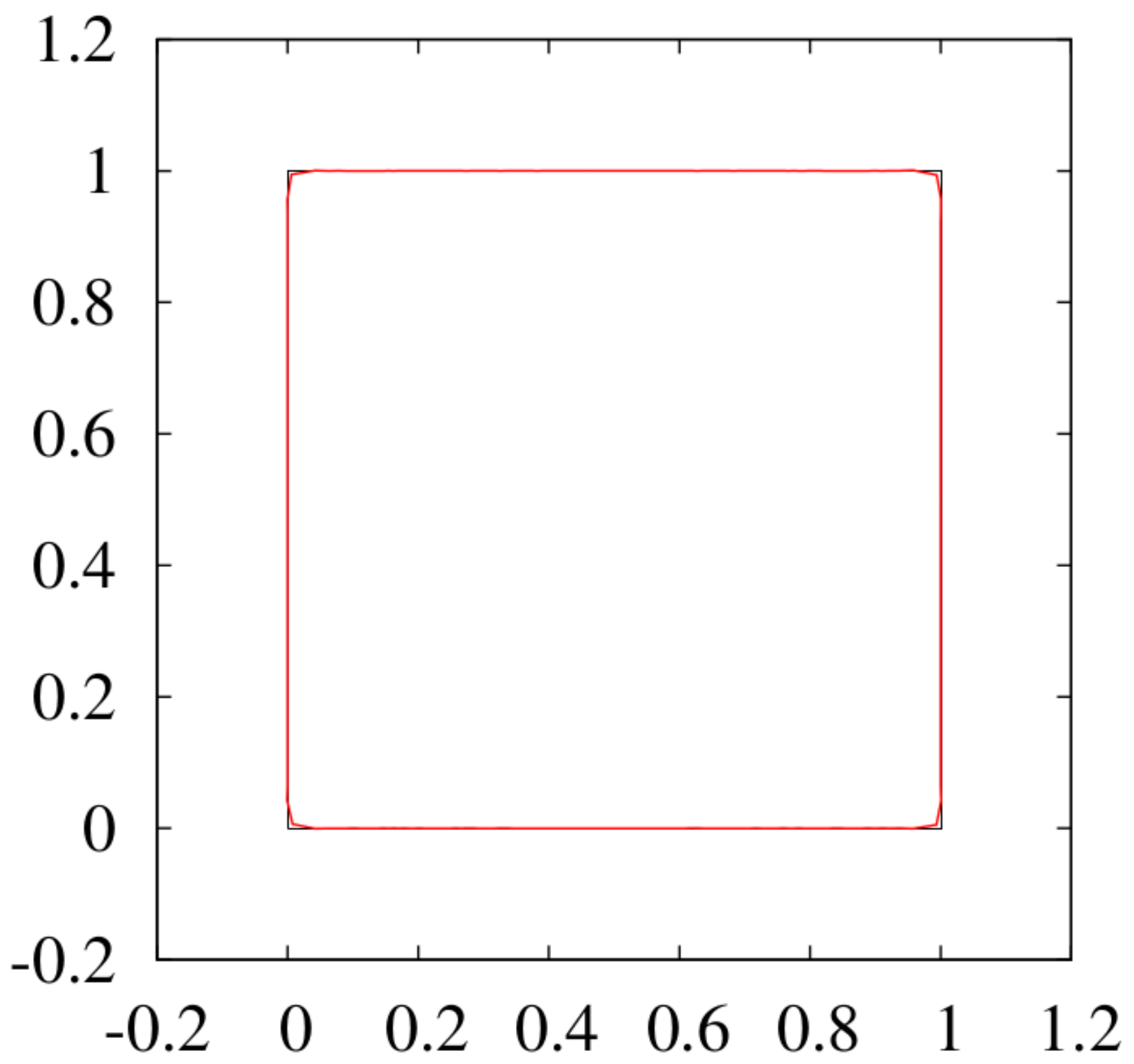}}
	\caption{$J_\infty^h$ with Euclidean metric}
    \end{subfigure}%
    \begin{subfigure}{.38\textwidth}
        \centering
	\scalebox{0.27}{\includegraphics{./plots/iter_L2_2000.pdf}}
	\caption{$J_2^h$ with  $H^1$-metric}
     \end{subfigure}

     \caption{Comparison of final shapes}\label{fig:comparison}
\end{figure}

All implementations were carried out within the FEniCS Software package \cite{fenics}.  The
quadratic program \eqref{eq:quadratic} is solved with the
python package \emph{cvxopt}; cf. \cite{cvxopt}.

\begin{figure}[h]
    
\begin{gnuplot}[terminal=epslatex,terminaloptions=color dashed]
	set xlabel 'iteration $n$ (log scale)'
set ylabel '$J_2^h(\Omega_n)$ (log scale)'
set logscale y
set logscale x
set xrange [1:2000]
plot 'cost_func_L20'  using 1:2 with lines lc rgb "blue" lt -1 lw 1 title '$J_\infty^h$ with Euclidean metric', \
     'cost_func_L20_H1'  using 1:2 with lines lc rgb "green" lt 2 lw 1 title '$J_\infty^h$ with $H^1$ metric',\
     'cost_func_L21' using 1:2 with lines lc rgb "red" lt -1 lw 1 title '$J_2^h$ with Euclidean metric',\
    'cost_func_L21_H1' using 1:2 with lines lc rgb "orange" lt 2 lw 1 title '$J_2^h$ with $H^1$ metric',\
    x**(-2) with lines lw 1 lt -1 title '$f(x)=x^{-2}$'
\end{gnuplot}
\caption{x-axis number of iterations and y-axis values $J_2^h(\Omega_n)$}
    \label{eq:cost_log_plot_L2_Linfty_euclidean}

\end{figure}

\bibliographystyle{plain}
\bibliography{refs}

\end{document}